\documentclass{article}
\usepackage[margin=1in]{geometry}
\usepackage{amsmath,amssymb,mathtools}
\usepackage{amsthm,mathrsfs}
\usepackage{graphicx}

\usepackage{booktabs}
\usepackage{multirow}
\usepackage{colortbl}
\usepackage{diagbox}
\usepackage{float}

\usepackage{makecell}
\usepackage{subcaption}
\usepackage{arydshln}
\usepackage{xargs}
\usepackage{float}
\usepackage{cases}
\usepackage{caption}
\usepackage{authblk}
\usepackage[hidelinks]{hyperref}
\newcommand{\va}{\boldsymbol{a}}

\newcommand{\ba}{\mathbf{a}}
\newcommand{\bb}{\mathbf{b}}

\newcommand{\bff}{\mathbf{f}}
\newcommand{\bg}{\mathbf{g}}

\newcommand{\br}{\mathbf{r}}

\newcommand{\bw}{\mathbf{w}}
\newcommand{\bx}{\mathbf{x}}
\newcommand{\by}{\mathbf{y}}

\newcommand{\bA}{\mathbf{A}}

\newcommand{\bF}{\mathbf{F}}

\newcommand{\bW}{\mathbf{W}}

\newcommand{\cB}{\mathcal{B}}

\newcommand{\cL}{\mathcal{L}}

\newcommand{\valpha}{{\boldsymbol{\alpha}}}
\newcommand{\vphi}{\boldsymbol{\phi}}
\newcommand{\vpsi}{\boldsymbol{\psi}}

\newcommand{\R}{\mathbb R}
\newcommand{\ds}{\displaystyle}

\definecolor{blue}{rgb}{0.0,0.0,1.0}
\definecolor{red}{rgb}{1.0,0.0,0.0}
\definecolor{purple}{rgb}{0.75, 0.0, 1.0}

\title{Generalized Transferable Neural Networks for Steady-State Partial
Differential Equations}
\date{}
\author[1]{Tao Cheng}
\author[2]{Lili Ju \thanks{\href{ju@math.sc.edu}{Corresponding author: ju@math.sc.edu}}}
\author[3]{Zhonghua Qiao}
\author[1]{Xiaoping Zhang}

\affil[1]{\small School of Mathematics and Statistics, Wuhan University, Wuhan, Hubei, 430072, China}
\affil[2]{\small Department of Mathematics, University of South Carolina, Columbia, SC 29208, USA}
\affil[3]{\small Department of Applied Mathematics, The Hong Kong Polytechnic University, Hung Hom, Kowloon, Hong Kong}

\newcommand\keywords[1]{\textbf{Keywords}: #1}

\newtheorem{lemma}{Lemma}
\newtheorem{theorem}{Theorem}
\newtheorem{remark}{Remark}

\begin{document}

\maketitle

\begin{abstract}
Deep learning has emerged as a compelling framework for scientific and engineering computing, motivating growing interest in neural network-based solvers for partial differential equations (PDEs). Within this landscape, network architectures with deterministic feature construction have become an appealing approach, offering both high accuracy and computational efficiency in practice. Among them, the transferable neural network (TransNet) is a special class of shallow neural networks (i.e., single-hidden-layer architectures), whose hidden-layer parameters are predetermined according to the principle of uniformly distributed partition hyperplanes. Although TransNet has demonstrated strong performance in solving PDEs with relatively smooth solutions, its accuracy and stability may deteriorate in the presence of highly oscillatory solution structures, where activation saturation and system conditioning issues become limiting factors. In this paper, we propose a generalized transferable neural network (GTransNet) for solving steady-state PDEs, which augments the original TransNet design with additional hidden layers while preserving its interpretable feature-generation mechanism. In particular, the first hidden layer of GTransNet retains TransNet's parameter sampling strategy but incorporates an additional symmetry constraint on the neuron biases, while the subsequent hidden layers omit bias terms and employ a variance-controlled sampling strategy for selecting neuron weights. This architecture effectively mitigates saturation effects and enhances adaptability to local solution variations, as supported by theoretical analysis and quantitative studies of neuron behaviors. A broad suite of numerical experiments and comparisons in two and three dimensions demonstrates that the proposed GTransNet significantly and consistently outperforms TransNet, particularly for PDEs with high-frequency and multiscale solutions. The proposed framework therefore advances the development of neural network methods by providing a powerful and efficient computational tool for steady-state PDE problems.
\end{abstract}


\section{Introduction}
Deep learning has achieved groundbreaking progress in many fields, including image recognition \cite{ResNet}, time series analysis \cite{wang2025timemixer}, computer vision \cite{Huang2017}, and generative modeling \cite{diffusionmodels}. Its success has also increasingly extended into scientific and engineering computing, fostering interdisciplinary collaboration among mathematicians, computer scientists, and engineers. A central focus of this effort lies in leveraging both data and physical knowledge to numerically solve partial differential equations (PDEs). Neural network–based PDE solvers, including deep Ritz method \cite{deepRitz}, deep Nitsche method \cite{MR4247207}, deep Galerkin method \cite{DGM}, physics-informed neural networks (PINNs) \cite{PINNs}, weak adversarial networks (WAN) \cite{WAN}, penalty-free neural networks (PFNN) \cite{SHENG2021110085}, and neural operators (FNO \cite{Li2020FourierNO}, DeepONet \cite{DeepONet}), have emerged as powerful computational frameworks. Among them, PINNs stand out by embedding governing equations and boundary conditions as physical constraints during training, offering advantages such as mesh-free implementation, flexible boundary handling, and a unified treatment of forward and inverse problems \cite{MR4229292,MR4083367,MR3995303,  MR4513793,  MR4337814, MR4166066}.

Compared with classical numerical schemes such as finite element method (FEM), finite volume method (FVM), finite difference method (FDM), and spectral method (SM), neural network–based approaches avoid mesh generation and are naturally applicable to high-dimensional problems \cite{WAN}. However, most deep learning-based solvers rely on stochastic optimization, and the nonconvexity of the objective or loss functions often leads to high computational cost and slow convergence \cite{Markidis2021TheOA}. While stochastic gradient descent (SGD) and its variants have gradually improved in efficiency, achieving highly accurate approximate solutions still remains very challenging in general.
Moreover, neural networks exhibit an intrinsic spectral bias (or F-principle) \cite{Rahaman2018OnTS, FPrinciple}, tending to learn low-frequency components first and making high-frequency features difficult to capture. To address this issue, several methods have been proposed to enhance high-frequency solution learning, such as the multiscale deep neural network (MscaleDNN) \cite{MscaleDNN2020}, which introduces multiple scaling factors in the frequency domain, and the feature-enhanced physics-informed radial basis neural networks (FE-PIRBN) \cite{ZHANG2025113798}, which combines PIRBN with hash encoding for high-frequency electromagnetic scattering. Further extensions and applications are presented in \cite{Xian2024,XiAn2023, Yong2024, MR4636189}.

Shallow neural networks, usually referring to fully connected (FC) networks with a single hidden layer, have recently regained attention due to their efficiency and accuracy. A representative class of such methods includes Extreme Learning Machine (ELM) \cite{ELM}, Random Feature Method (RFM) \cite{RFM}, and Transferable Neural Network (TransNet) \cite{TransNet}. Unlike conventional deep neural network models, these methods first predetermine the weights and biases of the hidden-layer neurons without training, while the weights of the output layer are computed via least-squares solvers, thereby achieving high computational efficiency.

Specifically, ELM typically samples neuron weights and biases from uniform or Gaussian distributions, and RFM further enhances expressive capacity by additionally incorporating domain decomposition and the partition-of-unity approach. These methods have been successfully applied to various PDE problems and low-dimensional applications, as documented in \cite{CHENJinrun25,Dong2021,DONG2021110585,ELIADEFALCO2026114553,Wangfei2024, ChenJingrun2025, Dong2024, Jingrun2025, weakTransNet}. TransNet employs a geometrically interpretable technique with location and shape parameters to predetermine neuron weights and biases, producing a uniform distribution of partition hyperplanes in the hidden layer and enabling more effective and robust feature extraction than ELM and RFM in practice \cite{TransNet}.
Several variants of TransNet have also been developed, including the Multi-TransNet method for elliptic interface problems \cite{Lu25}, which outperforms RFM \cite{Chi24} and ELM \cite{Liang25} and traditional finite element methods, and the MAE-TransNet method for singularly perturbed problems \cite{Shen2025}, which effectively overcomes boundary layer difficulties.

Although TransNet, through its adaptive selection of shape parameters, has the potential to handle problems with oscillatory solutions, it suffers from saturation effects that require a large number of neurons to maintain accuracy, thereby severely increasing the condition number of the resulting least-squares systems and reducing computational efficiency. To overcome these limitations, we propose in this paper a generalized transferable neural network (GTransNet) with multiple hidden layers for solving steady-state PDEs. The proposed method retains the structural design of TransNet in the first hidden layer and adopts no bias terms together with a variance-controlled weight sampling strategy in the subsequent hidden layers. The GTransNet framework significantly improves the network’s expressive capacity and adaptability, effectively mitigating TransNet’s performance degradation in high-frequency solution regimes. Both theoretical analysis and numerical experiments confirm that GTransNet achieves superior accuracy and stability across a wide range of challenging scenarios in science and engineering applications.


The remainder of this paper is organized as follows. Section \ref{sec:2} provides an overview of shallow neural network–based PDE solvers and highlights key properties of TransNet, including a qualitative analysis of its neuron behavior. Section \ref{sec:3} introduces the proposed GTransNet architecture and presents the associated theoretical results. Section \ref{sec:4} reports a series of numerical experiments in two and three dimensions, including Poisson equation, 
Navier-Stokes equations, Helmholtz equation, multiscale elliptic equation and Allen-Chan equation,  to illustrate the accuracy, efficiency, and robustness of GTransNet. Finally, Section \ref{sec:5} offers concluding remarks and discusses potential directions for future research.

\section{Transferable Neural Networks for PDEs}\label{sec:2}

Let us consider the following general form of a steady-state PDE problem:
\begin{equation}\label{modelp}
\begin{cases}
\mathcal{L}(u)(\bx) = f(\bx), & \bx \in \Omega, \\
\mathcal{B}(u)(\bx) = g(\bx), & \bx \in \partial\Omega,
\end{cases}
\end{equation}
where $\Omega \subset \mathbb{R}^d$ is an open bounded spatial domain with a Lipschitz boundary \(\partial \Omega\),
$u(\bx):\Omega\rightarrow\mathbb{R}$ denotes the unknown function,
\(\mathcal{L}(\cdot)\) defines a linear or nonlinear differential operator, and \(\mathcal{B}(\cdot)\) specifies 
certain boundary conditions. The functions  \(f(\bx)\) and \(g(\bx)\) represent the source term and the prescribed boundary data, respectively. We also note that when the target PDE problem is time-dependent, the temporal variable $t$ often can be combined with the spatial variable $\bx$ for an unified processing in neural networks.

\subsection{Shallow neural network-based solvers for PDEs}
A shallow neural network $u_{\mathrm{SNN}}: \Omega \rightarrow \mathbb{R}$ with a single hidden layer can be represented as
\begin{equation}\label{snn}
u_{\mathrm{SNN}}(\bx) = \valpha^\top \vpsi\quad\mbox{with}\quad  \vpsi = \sigma(\bW \bx + \bb), 
\end{equation}
where $\bx\in \Omega$ is the network input, $\bW = [\bw_1\, \bw_2\, \ldots\, \bw_M]^\top \in \R^{M\times d}$ and $\bb = (b_1, b_2, \ldots, b_M) \in \R^{M}$ are respectively referred to as the neuron weights and biases of the hidden layer, with $M$ denoting the number of neurons, $\sigma$ is an activation function, and $\valpha = (\alpha_1, \ldots, \alpha_M) \in \R^{M}$ represents the weights of the output layer. Thus, the hidden-layer neurons of the shallow neural network \eqref{snn} are given by $\vpsi(\bx) = (\psi_1(\bx), \psi_2(\bx), \ldots, \psi_M(\bx))$, with
\begin{equation}\label{eq:psi}
\psi_m(\bx) = \sigma(\mathbf{w}_m^\top \bx + b_m), \quad m = 1,\cdots, M.
\end{equation}
Note that the bias is omitted in the output layer in this case, since it is usually not very useful in practice.

Given a set of collocation points $\{\bx_{i}\}_{i=1}^{K_{\text{int}}}$ sampled inside the domain $\Omega$ and $\{\mathbf{y}_{j}\}_{j=1}^{K_{\text{bdy}}}$ on the  boundary $\partial \Omega$, a shallow neural network-based solver for the PDE problem \eqref{modelp} can then be formulated as a residual-type loss minimization problem as follows:  find $u_{\mathrm{SNN}}$ defined by \eqref{snn} such that
\begin{equation}\label{eq:optim_SN-o}
\begin{aligned}
&\min_{\bW,\bb,\valpha} \;{\cal F}_{Loss}(u_{\mathrm{SNN}};\,\{\bW,\bb,\valpha\}) \\
&\qquad := \textstyle \;\lambda_1\sum\limits_{i=1}^{K_{\text{int}}} \|\cL(u_{\mathrm{SNN}})(\bx_i)-f(\bx_i)\|^2 
+ \lambda_2 \sum\limits_{j=1}^{K_{\text{bdy}}} \|\cB(u_{\mathrm{SNN}})(\mathbf{y}_j)-g(\mathbf{y}_j)\|^2,
\end{aligned}
\end{equation}
where $\lambda_1,\lambda_2>0$ are some user-defined penalty parameters used  to balance the interior and boundary residuals, and $\cL(\vpsi)$ and $\cB(\vpsi)$ denote the elementwise application of the operators $\cL$ and $\cB$ to the hidden-layer vector $\vpsi$. 

Among existing shallow neural network–based PDE solvers, a popular approach is to randomly predetermine the weights $\bW$ and the biases $\bb$ of the hidden-layer neurons and to learn only the weights $\valpha$ of the output layer. That is, the original minimization problem \eqref{eq:optim_SN-o} becomes a smaller-scale one: given ${\bW,\bb}$, find $u_{\mathrm{SNN}}$ such that
\begin{equation}\label{eq:optim_SN}
\min_{\valpha}\; {\cal F}_{Loss}(u_{\mathrm{SNN}};\,\{\valpha\}).
\end{equation}
Such strategy has been widely adopted in models such as RFM \cite{RFM}, ELM \cite{ELM} and TransNet \cite{TransNet}. Denote $K := K_{\mathrm{int}} + K_{\mathrm{bdy}}$ as the total number of training points. 
For linear PDE problems, the solution of the optimization problem \eqref{eq:optim_SN} is then equivalent to solving the following linear system of size $K\times M$ in the least squares sense:
\begin{equation}\label{eq:ls}
    \begin{bmatrix}
        \lambda_1\bF_{\mathrm{pde}}   \\
        \lambda_2\bF_{\mathrm{bdy}}
    \end{bmatrix}
    \valpha = \begin{bmatrix}
        \lambda_1 \bff \\ 
        \lambda_2 \bg
    \end{bmatrix},
\end{equation}
where 
\begin{equation}\label{eq:FFfg}
\begin{array}{c}
\bF_{\mathrm{pde}} = \begin{bmatrix}
\cL(\vpsi^\top)(\bx_1) \\
\vdots \\
\cL(\vpsi^\top)(\bx_{K_{\text{int}}}) 
\end{bmatrix}, ~
\bF_{\mathrm{bdy}} = \begin{bmatrix}
\cB(\vpsi^\top)(\mathbf{y}_1) \\
\vdots \\
\cB(\vpsi^\top)(\mathbf{y}_{K_{\text{bdy}}}) 
\end{bmatrix},  \\ \\
\bff = \begin{bmatrix}
f(\bx_1) \\
\vdots \\
f(\bx_{K_{\text{int}}}) 
\end{bmatrix}, ~
\bg = \begin{bmatrix}
g(\mathbf{y}_1) \\
\vdots \\
g(\mathbf{y}_{K_{\text{bdy}}}) 
\end{bmatrix}.
\end{array}
\end{equation}
The penalty parameters can be selected as follows: 
\begin{equation}\label{eq:w_p}
\begin{aligned}
&\ds \lambda_1 =\frac{1}{\max_{1\le i \le K_{\mathrm{int}}}\bigl(\max_{1\le j \le M}|\cL(\psi_j)(\bx_i)|,\ |f(\bx_i)|\bigr)},\\
&\ds \lambda_2 = \frac{1}{\max_{1\le i \le K_{\mathrm{bdy}}}\bigl(\max_{1\le j \le M}|\cB(\psi_j)(\mathbf{y}_i)|,\ |g(\mathbf{y}_i)|\bigr)},
\end{aligned}
\end{equation}
which have been discussed in \cite{Chi24,Lu25}. For nonlinear PDE problems, one can linearize the optimization problem \eqref{eq:optim_SN} and then effectively solve it using the Picard iteration approach; see \cite{TransNet} for details.

\subsection{Partition hyperplanes and sampling strategy in TransNet}

The key difference between ELM, RFM, and TransNet lies in the predetermination strategy of the single hidden-layer parameters. While ELM and RFM typically sample the weights from classical uniform or Gaussian distributions, TransNet employs a more meaningful sampling scheme designed to achieve better numerical approximation ability \cite{TransNet}.

Inspired by the activation patterns of ReLU networks, TransNet reparameterizes each neuron of the hidden layer defined in~\eqref{eq:psi} as
\begin{equation}\label{eq:TransNet_Init}
\psi_m(\bx) = \sigma\big(\gamma_m (\mathbf{a}_m^\top \bx + r_m)\big),\quad m=1,\ldots,M,
\end{equation}
where the unit vector $\mathbf{a}_m$ (i.e., $|\ba_m|=1$) and the scalar $r_m$ serve as location parameters that define
the direction and offset of the $m$-th partitioning hyperplane, respectively, and
the parameter $\gamma_m>0$ is a predefined shape factor controlling the steepness of the activation function.
This implies $\ba_m=\bw_m/\gamma_m$ and $r_m = b_m/\gamma_m$ in~\eqref{eq:psi}.

Each hidden-layer neuron can be viewed as a hyperplane $\mathbf{a}_m^\top \bx + r_m = 0$ dividing the input space into two regions. In ReLU neural networks, the intersections of these hyperplanes form piecewise linear regions, within which the network provides a linear approximation of the target function. When solving PDEs with smooth solutions, simply increasing the number of linear regions does not guarantee overall higher accuracy; their spatial uniformity also plays a crucial role. This insight motivated the concept of ``uniform neuron distribution'' in \cite{TransNet}, which enhances the network’s global approximation capability. For completeness, we next introduce the concept and construction of uniform neuron distribution as presented in \cite{TransNet}.

Taking the unit ball $B_1(\mathbf{0}) \subset \mathbb{R}^d$ as an example, to achieve uniform expressivity across this domain, the TransNet method \cite{TransNet} predetermines and distributes the hyperplanes $\mathbf{a}_m^\top \bx + r_m = 0$ in the following way. Each neuron corresponds to a partition hyperplane, whose direction is randomly sampled on the unit sphere $S^{d-1}$:
\begin{equation}\label{eq:sample_a}
\mathbf{a}_m = \frac{\mathbf{z}}{|\mathbf{z}|}, \quad \mathbf{z} \sim \mathcal{N}(\mathbf{0}, {\bf I}_d),
\end{equation}
that is, a vector is drawn from a standard Gaussian distribution and then normalized to unit length. The absolute offset $|r_m|$ determines the distance of the hyperplane from the origin and is sampled according to
\begin{equation}\label{unif1}
r_m \sim \mathcal{U}[0, 1]
\end{equation}
as suggested in \cite{TransNet}.

\begin{remark}
Due to the symmetry of $|r_m|$,
we also note that an alternative way to sample $r_m$ is to use the uniform distribution over $[-1,1]$, i.e.,
\begin{equation}\label{unif2}
r_m \sim \mathcal{U}[-1, 1],
\end{equation}
which additionally maintains the central symmetry of the neuron bias distribution about zero.
\end{remark}

For any point $\by \in B_1(\mathbf{0})$, the distance to the hyperplane corresponding to the $m$-th neuron is clearly given as
$\text{dist}(\by, m) = \big|\ba_m^\top (\by - r_m \ba_m)\big|$.
Based on this definition, a density function is introduced in \cite{TransNet} to quantify the spatial distribution of partition hyperplanes, or equivalently, hidden-layer neurons:
\begin{equation}
D_M(\by) = \frac{1}{M} \textstyle\sum\limits_{m=1}^{M} \mathbf{1}_{\{\text{dist}(\by, m) < \tau\}}(\by),
\end{equation}
where $\mathbf{1}_{\{\cdot\}}$ denotes the indicator function. The following theoretical result characterizes the distribution of hidden-layer neurons in TransNet.

\begin{theorem}[Uniform Neuron Distribution in $B_1(\mathbf{0)}$ \cite{TransNet}]\label{thm:uniform_neuron}
If the direction vectors $\{\ba_m\}_{m=1}^M$ are uniformly distributed on the unit sphere $S^{d-1}$ and the offset parameters $\{r_m\}_{m=1}^M$ are independent and uniformly distributed over the unit interval $[0,1]$, then for any point $\by\in B_1(\mathbf{0)}$ satisfying $\|\by\|_2 \leq 1 - \tau$, it holds
\begin{equation}
\mathrm{E}[D_M(\by)] = \tau.
\end{equation}
\end{theorem}

According to Theorem \ref{thm:uniform_neuron}, the hidden-layer neurons in TransNet are expected to be uniformly distributed within the unit ball $B_1(\mathbf{0})$ (slightly away from $\partial B_1(\mathbf{0})$).
Through an affine transformation (including scaling and translation), this distribution can be extended to a ball $B_R(\mathbf{x_c})$ centered at $\mathbf{x_c}$ with radius $R$.
In this case, the corresponding partition hyperplane is given by
$\mathbf{a}_m^\top (\mathbf{x - x_c}) + R r_m = 0.
$
For a general problem domain $\Omega$, the center $\mathbf{x_c}$ of the ball $B_R(\mathbf{x_c})$ should be chosen as approximately the geometric center of $\Omega$ and the radius $R$ should be chosen to approximate the geometric center of $\Omega$, and the radius $R$ should be selected such that $B_R(\mathbf{x_c})$ is slightly larger than the domain $\Omega$
(i.e., $\Omega \subset B_R(\mathbf{x_c})$), in order to better enforce the boundary conditions and enhance the effectiveness of the hidden-layer neurons.
For further details, see \cite{Lu25}.

\begin{remark}
It is worth noting that while ELM and RFM can be straightforwardly extended to the setting of multi-hidden-layer neural networks, this extension remains challenging for TransNet due to the involvement of domain geometry information and the introduction of uniform neuron distribution.
\end{remark}

\subsection{Quantitative study of neuron behaviors in TransNet}

For simplicity, it is common in practice to take the same shape parameter for all hidden-layer neurons, i.e., $\gamma_m = \gamma$ for $m = 1,\ldots,M$. In addition, using ReLU as the activation function in TransNet often yields low accuracy, likely due to its weak nonlinearity; instead, the activation function $\sigma = \tanh$ is more commonly used for PDE-related problems. We follow these practices for analysis in this work, i.e.,
\begin{equation}\label{ssp}
\psi_m(\bx) = \tanh\big(\gamma (\mathbf{a}_m^\top (\bx-\bx_c) + Rr_m)\big),\quad m=1,2,\ldots,M.
\end{equation}

The shape parameter $\gamma$ plays a crucial role in the performance of TransNet for solving PDEs. When the PDE solution is relatively smooth, the model is less sensitive to $\gamma$, and a small value of $\gamma$ can readily yield satisfactory results. However, when the solution contains high-frequency components or steep gradients, the model’s performance becomes highly sensitive to variations in $\gamma$. Notably, experimental results indicate that an excessively large $\gamma$ does not necessarily improve accuracy; instead, it may even degrade performance or lead to solver failure. To investigate this phenomenon, we conduct some qualitative analysis of the behavior of hidden-layer neurons.

Figure \ref{fig:transnet_basis} illustrates visualizations over the square $(-1,1)^2$ of a sample hidden-layer neuron $\psi_m$ in TransNet associated with $B_{1.5}(\mathbf{0})$ (i.e., $\bx_c=\mathbf{0}$ and $R=1.5)$ in two dimensions under different shape parameters. Given sampled parameters $\va_m$ and $r_m$ through \eqref{eq:sample_a} and \eqref{unif1}, they together define a hyperplane, where $\va_m^\top \bx + R r_m$ represents the signed distance from a point $\bx$ to that hyperplane. This distance is then scaled by the shape parameter $\gamma$ and passed through the activation function $\tanh$ to produce the neuron response. It is easy to see that, as $\gamma$ increases, the activations enter the saturation region (where the output of the $\tanh$ function approaches its limits (near $\pm 1$) and becomes nearly insensitive to input changes) more rapidly. Although a larger $\gamma$ can yield steeper gradients within the non-saturated region and enhance sensitivity to local features, the shrinking of this region weakens the model’s overall expressive capacity, thereby reducing its global approximation ability. Figure \ref{fig:density_transnet} presents the corresponding distribution histograms of the hidden-layer neuron activation values under different values of $\gamma$. As $\gamma$ increases, most activation values cluster in the saturation region. Specifically, since TransNet uses uniform sampling for the offset parameter $r_m$ of the hidden-layer neurons over the range $[0,1]$, more neurons tend to accumulate near the value $1$, which further reflects the model’s limited expressive capability.

\begin{figure}[ht!]
    \centering
    \includegraphics[width=0.95\linewidth]{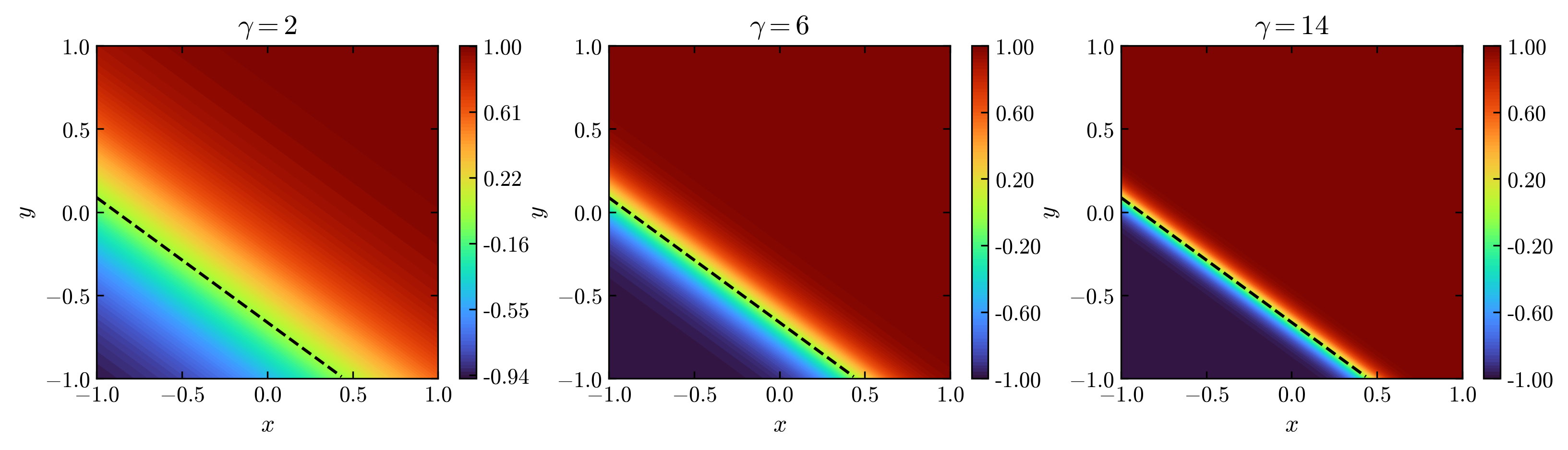}\vspace{-0.2cm}
    \caption{Visualization over the square $(-1,1)^2$  of a sample hidden-layer neuron $\psi_m(\bx)$ in TransNet associated with $B_{1.5}(\mathbf{0})$ in two dimensions, where the dashed line indicates the corresponding partition hyperplane. From left to right:  $\gamma$ =  2, 6, 14.}
    \label{fig:transnet_basis}\vspace{-0.3cm}
\end{figure}

\begin{figure}[ht!]
    \centering
    \includegraphics[width=0.9\linewidth]{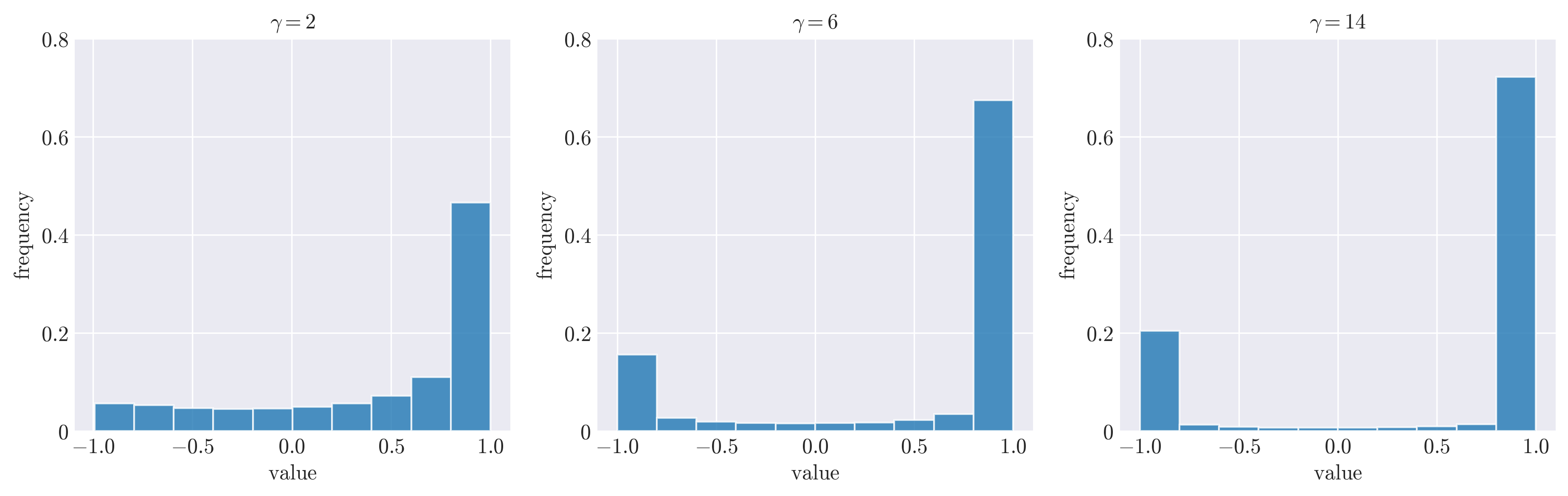}\vspace{-0.2cm}
    \caption{Distribution histograms of the hidden-layer neuron activation values in TransNet associated with $B_{1.5}(\mathbf{0})$ in two dimensions.  Activation values are collected from 500 input points over the unit square $(-1,1)^2$, with each point corresponding to 1000 different hidden-layer neurons, illustrating how the response pattern and concentration vary along with the shape parameter $\gamma$.  From left to right: $\gamma$ =  2, 6, 14.}
    \label{fig:density_transnet}\vspace{-0.3cm}
\end{figure}

\begin{figure}[ht!]
    \centering
    \includegraphics[width=0.65\linewidth]{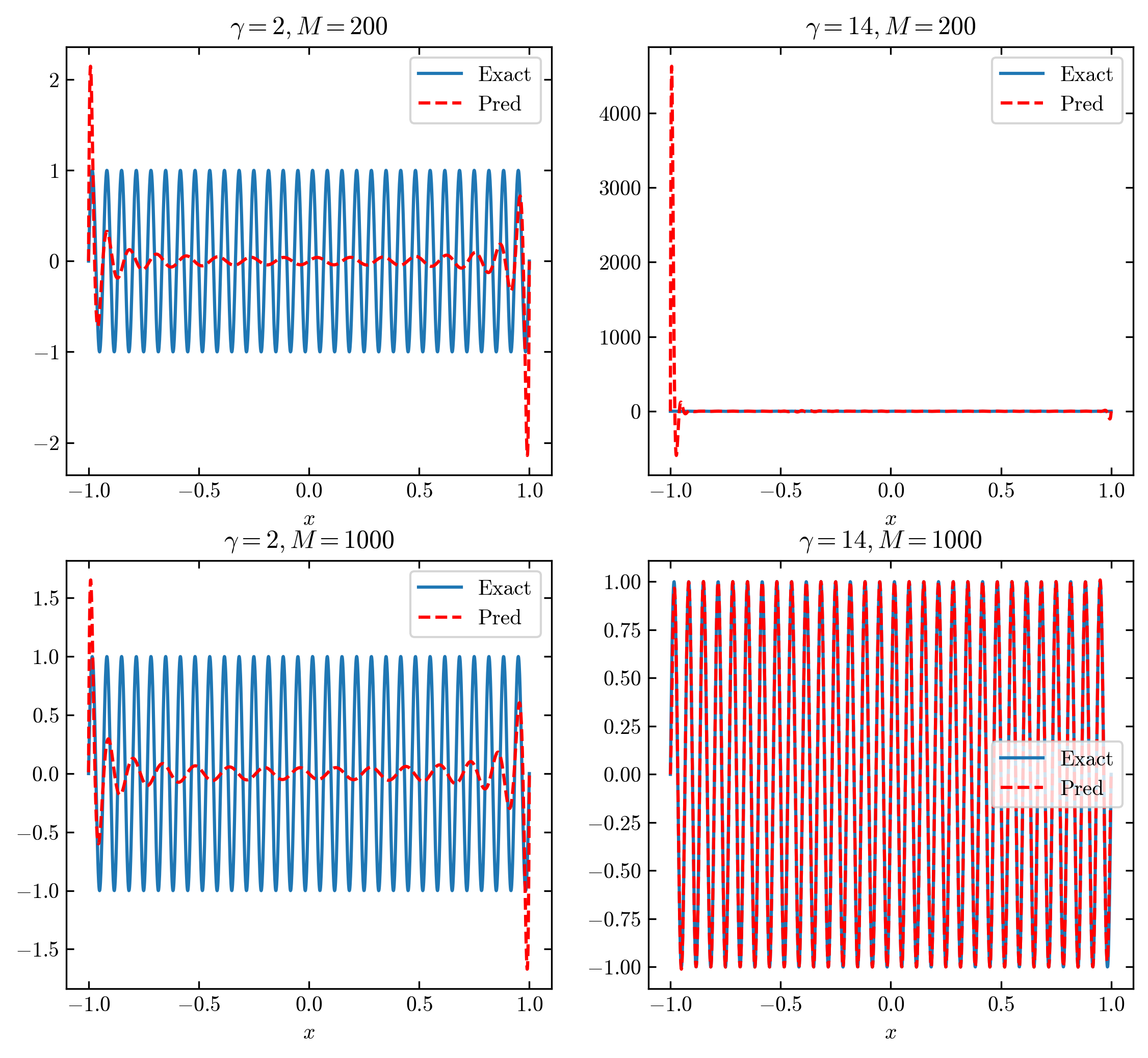}
    \vspace{-0.2cm}
    \caption{Numerical results of fitting the high-frequency function $f(x) = \sin(30\pi x)$ in the interval $(-1,1)$ by using a
    TransNet associated with $B_{1.1}(\mathbf{0})$ in one dimension. 1000 uniformly distributed collocation points are used. From Left to right: $\gamma=2, 14$; From top to bottom:  $M= 200, 1000$.}
    \label{fig:1Dtransnet}\vspace{-0.3cm}
\end{figure}

Based on these findings and analyses, we can understand why TransNet performs well for problems with smooth solutions, but its performance deteriorates when dealing with problems with highly oscillatory or large-gradient solutions. To overcome this shortcoming, a natural approach is to use a large number of hidden-layer neurons (i.e., large values of $M$) together with a large value of the shape parameter $\gamma$. This approach is feasible to some extent, as demonstrated by the following example of fitting a one-dimensional function. We use a TransNet to fit the high-frequency function $y = \sin(30\pi x)$ on the interval $(-1,1)$, and we choose the ball $B_{1.1}(\mathbf{0})$ (i.e., $\bx_c = 0$ and $R = 1.1$) to generate the hidden-layer neurons. Figure \ref{fig:1Dtransnet} shows the corresponding numerical results with different numbers of hidden-layer neurons and different values of the shape parameter. It is observed that when $\gamma$ or $M$ is small (i.e., $\gamma = 2$ or $M = 200$), the fitting accuracy is poor. When $\gamma$ and $M$ are sufficiently large (i.e., $\gamma = 14$ and $M = 1000$), TransNet accurately approximates the target function. However, it is clear that such computational scaling is very expensive and may even be unfeasible for higher-dimensional problems.

\section{Generalized Transferable Neural Networks}\label{sec:3}
To overcome the limitations of TransNet in handling high-frequency and large-gradient solutions, we propose an {\em a multi-hidden-layer version} of TransNet, called {``Generalized Transferable Neural Networks''} (GTransNet), which are aimed at improving and enhancing the applicability and robustness in solving steady-state PDE problems.

\subsection{Construction of GTransNet and  structure analysis}
Compared to TransNet, our proposed GTransNet incorporates additional hidden layers, consisting of two or more hidden layers. The first hidden layer basically follows the same sampling scheme for neuron weights as in TransNet, but the neuron biases are instead drawn from a centrally symmetric distribution. For the subsequent hidden layers, the weight parameters are predetermined through a carefully designed sampling strategy with variance control, and no bias terms are applied. The weights of the output layer are obtained in the same manner as in TransNet, by minimizing the residual-type loss function. Let us assume that the domain $\Omega \subset B_R(\mathbf{x_c})$.

The overall structure of GTransNet with $L\ge 2$ hidden layers (i.e., a fully connected neural network of $L+1$ layers) is formulated as follows:
\begin{equation}\label{eq:g_transnet}
    \begin{cases}
        \vpsi_1 = \tanh({\mathbf\Gamma}\odot(\bA(\mathbf{x}-\mathbf{x}_c) + R\br)), \\
        \vpsi_2 = \tanh(\bW_2\vpsi_1), \\
        \vdots\\
        \vpsi_L = \tanh(\bW_L\vpsi_{L-1}), \\
        u_{\mathrm{NN}}(\mathbf{x}) = \valpha^\top \vpsi_L,
    \end{cases}
\end{equation}
where 
$\vpsi_1=(\psi^{(1)}_1,\psi^{(2)}_1,\cdots,\psi^{(1)}_{N_1})\in \R^{N_1}$, $\bA = [\ba_1\, \ba_2\,\ldots\, \ba_{N_1}]^\top\in\R^{{N_1}\times d}$, ${\mathbf\Gamma}=(\gamma_1,\gamma_2,\cdots,\gamma_{N_1})\in \R^{N_1}$,  $\br = (r_1, r_2\ldots, r_{N_1})\in\R^{{N_1}}$, $\odot$ denotes the elementwise multiplication, and for $l=2, 3,\ldots, L$, 
$$
\vpsi_l=(\psi^{(l)}_1,\cdots,\psi^{(l)}_{N_l})\in \R^{N_l}, \quad
\bW_l=(W^{(l)}_{ij})\in \R^{N_l\times N_{l-1}}, 
$$ 
where $N_l$ denotes the number of neurons in the $l$-th hidden layer. Figure \ref{fig:GTransNet} illustrates the network architecture of the proposed GTransNet. 
Note that no bias terms are used in the second and subsequent hidden layers.

\begin{figure}[ht!]
    \centering
    \includegraphics[width=0.8\linewidth]{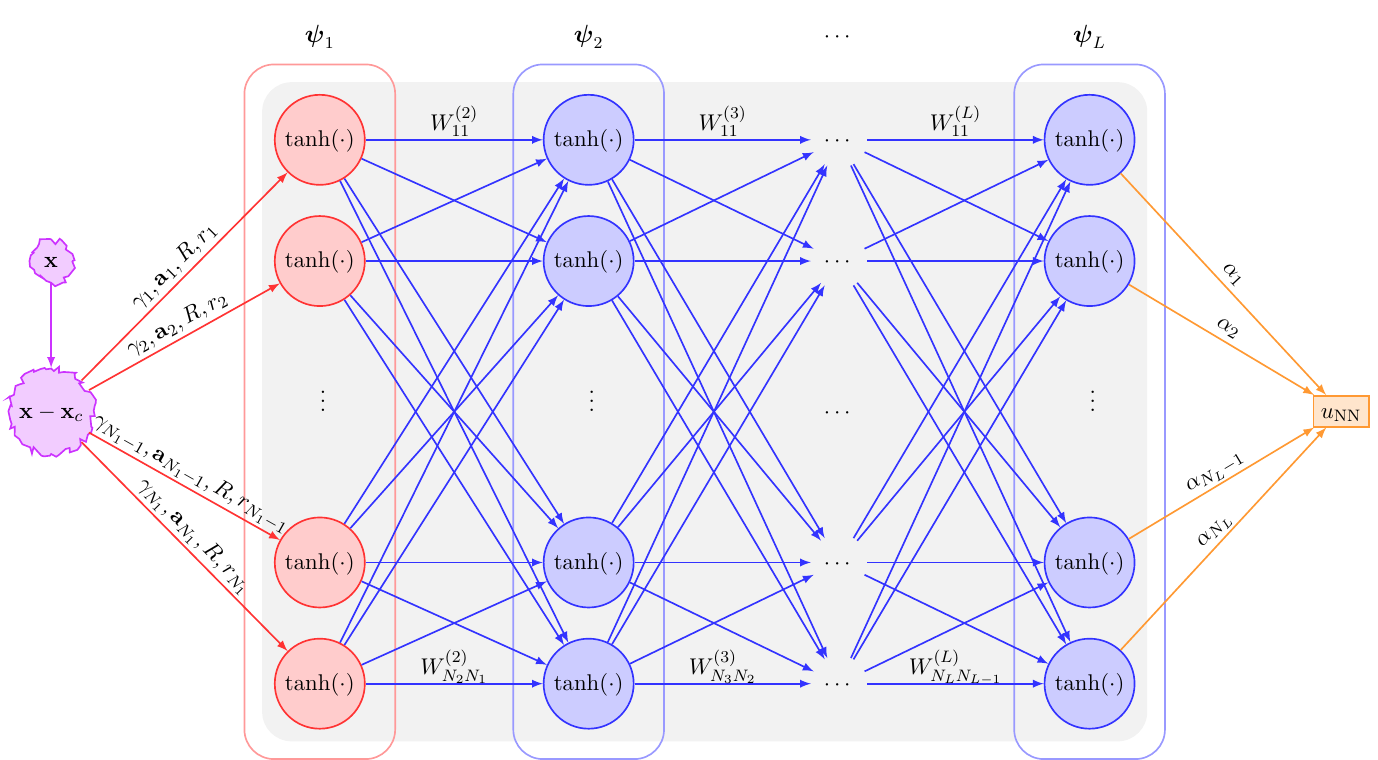}
     \vspace{0.1cm}
    \caption{Network architecture of the proposed GTransNet \eqref{eq:g_transnet} with $L$ hidden layers.}\vspace{-0.3cm}
    \label{fig:GTransNet}
\end{figure}

The GTransNet-based solver for the PDE problem \eqref{modelp} then leads to the following loss minimization problem: Given $\{{\mathbf\Gamma}, \bA, \br, \{\bW_i\}_{i=2}^{L}\}$ (i.e., the network parameters for all the $L$ hidden layers), find $u_{\mathrm{NN}}$ defined by \eqref{eq:g_transnet} such that
\begin{equation}\label{eq:optim_NN} \min_{\valpha}\; {\cal F}_{Loss}(u_{\mathrm{NN}};\,\{\valpha\}), \end{equation}
which is similar to \eqref{eq:optim_SN} for the TransNet solver.
For linear PDE problems, the solution of the optimization problem \eqref{eq:optim_NN} is again correspondingly equivalent to solving a linear system of size $K \times N_L$ in the least-squares sense similar to \eqref{eq:ls}.

In the first hidden layer of GTransNet defined in \eqref{eq:g_transnet}, we let the predetermination of the network parameters follow the sampling algorithms for TransNet introduced in the previous section. In particular, we use \eqref{eq:sample_a} to select the direction vectors $\{\ba_m\}_{m=1}^{N_1}$ and \eqref{unif2} for selecting the offset vector $\{r_m\}_{m=1}^{N_1}$ (i.e., with enforcement of central symmetry about zero).
Let us use the notation $X \stackrel{d}{=} Y$ to denote equality in distribution for the random variables $X$ and $Y$. It is easy to see that $\ba_m$ and $r_m$ are independent and satisfy
\begin{equation}\label{eq:CS}
\ba_m \stackrel{d}{=} -\ba_m, \qquad r_m \stackrel{d}{=} -r_m,
\end{equation}
for each $m = 1, 2, \cdots, N_1$.

\begin{lemma}[Zero Mean of the First Hidden Layer]\label{lem:1}
For any fixed $\bx \in \mathbb{R}^d$, the random variable $\psi^{(1)}_m$ defined in \eqref{eq:g_transnet} satisfies
\begin{equation}\label{eq:CS_psi}
\psi^{(1)}_m(\bx) \stackrel{d}{=} -\psi^{(1)}_m(\bx),\qquad m=1,2\cdots, {N_1},
\end{equation}
and consequently, 
\begin{equation}
\mathbb{E}[\vpsi_1(\bx)] = \mathbf{0}.
\end{equation}
\end{lemma}
\begin{proof}
Since $\ba_m$ and $r_m$ are independent and each is centrally symmetric about zero, their joint distribution satisfies
\begin{equation}\label{eq:CS1}
(\ba_m, r_m) \stackrel{d}{=} (-\ba_m, -r_m).
\end{equation}
Define the Borel measurable function $g(\ba_m,r_m)= \tanh(\gamma_m(\ba_m^\top \bx + Rr_m))$. Using the oddness of $\tanh$, we have the pointwise identity
\begin{equation}
\begin{aligned}
g(-\ba_m, -r_m) &= \tanh(\gamma_m((-\ba_m)^\top \bx + R(-r_m))) 
\\
&= -\tanh(\gamma_m(\ba_m^\top \bx + Rr_m)) = -g(\ba_m,r_m).
\end{aligned}
\end{equation}
By the standard property that equality in distribution is preserved under Borel measurable functions, we obtain
\[
\psi^{(1)}_m(\bx) = g(\ba_m,r_m) \stackrel{d}{=} g(-\ba_m,-r_m) = -\psi^{(1)}_m(\bx),
\]
thus $\psi_m(\bx)$ is centrally symmetric about zero. 
Since $\tanh$ is bounded, the expectation exists and satisfies
\[
\mathbb{E}[\psi^{(1)}_m(\bx)] = \mathbb{E}[-\psi^{(1)}_m(\bx)] = -\mathbb{E}[\psi^{(1)}_m(\bx)],
\]
and consequently  $\mathbb{E}[\psi^{(1)}_m(\bx)] = 0$, which immediately gives $\mathbb{E}[\vpsi(\bx)] = \mathbf{0}$ by taking $m$ from 1 to ${N_1}$.
\end{proof}

Next, we discuss the sampling strategy for the weight matrices $\{\bW_l\}_{l=2}^{L}$ in the subsequent hidden layers of GTransNet \eqref{eq:g_transnet}, which is partially inspired by the Xavier method \cite{Xavier}. The Xavier method is widely used for weight initialization in neural networks. It adjusts the initial values of the weights based on the number of neurons in the input and output layers to ensure that the signal variance remains consistent across each layer during both forward and backward propagation.
Specifically, for $l = 2, \dots, L$, we sample the weight matrices independently according to the following Gaussian distribution:
\begin{equation}\label{eq:sample_W}
W_{ij}^{(l)} \sim \mathcal{N} \left(0,\sigma_l^2\right)
\quad{\rm with}\;\;
\sigma_l = \textstyle\sqrt{\frac{\delta}{N_{l-1}}},
\end{equation}
where $0 < \delta \le 1$ is a variance control parameter, $\sigma_l^2$ represents the weight variance at $l$-th layer, and $N_{l-1}$ is the number of nodes in the $(l-1)$-th layer.

Let us denote the pre-activation $\vphi_l=\bW_l\vpsi_{l-1}$ for $l\ge 2$ then $\vpsi_l=\tanh(\vphi_l)$. 
We then have the following result regarding the GTransNet architecture  \eqref{eq:g_transnet} with the network parameters of its hidden layers predetermined through 
\eqref{eq:sample_a},  \eqref{unif2} and \eqref{eq:sample_W}.

\begin{theorem}[Zero Mean Propagation in GTransNet]\label{lem:CS} 
The means of  neuron pre-activations and activations in  the subsequent hidden layers of GTransNet \eqref{eq:g_transnet}  are all zeros, i.e.,
\begin{equation}\label{eq:E_phi}
\mathbb{E}[\vphi_l] = \mathbf{0}, \qquad l=2,\dots,L,
\end{equation}
and
\begin{equation}\label{eq:E_psi}
\mathbb{E}[\vpsi_l] = \mathbf{0}, \qquad l=2,\dots,L.
\end{equation}
\end{theorem}

\begin{proof} 
Note Lemma~\ref{lem:1} gives 
$\mathbb{E}[\vpsi_1] = \mathbf{0}.$
Based on the definition of the weight matrices $\{\bW_l\}_{l=2}^{L}$ in \eqref{eq:sample_W}, it is easy to find that 
$\bW_l$ is centrally symmetric about $\mathbf 0$ (i.e., $\bW_l \stackrel{d}{=} -\bW_l$)
and independent of the random vector $\vpsi_{l-1}$, consequently the joint distribution satisfies
\[
(-\bW_l, \vpsi_{l-1}) \stackrel{d}{=} (\bW_l, \vpsi_{l-1}).
\] 
By the property that equality in distribution is preserved under Borel measurable functions and the function $f(\bW_l, \vpsi_{l-1}) = \bW_l \vpsi_{l-1}$ is a Borel measurable function, it follows that
\[
f(\bW_l, \vpsi_{l-1}) \stackrel{d}{=}  f(-\bW_l, \vpsi_{l-1}),
\]
implying that $\bW_l \vpsi_{l-1}$ is centrally symmetric about $\mathbf{0}$, 
i.e., 
\begin{equation*}
\vphi_{l} \stackrel{d}{=} -\vphi_{l}, \quad l=2,\cdots,L,
\end{equation*}
which immediately implies \eqref{eq:E_phi}. 
Since $\bW_l\vpsi_{l-1}$ is centrally symmetric about zero and $\tanh$ is an odd function, we further obtain \eqref{eq:E_psi} and complete the proof. 
\end{proof}

\begin{theorem}[Controlled Variance Propagation in GTransNet]\label{thm:1}
Assume that each component of $\vpsi_1$ has a variance not more than $\sigma_0^2$. 
Then the variances of the neuron activations  in  the subsequent hidden layers of GTransNet \eqref{eq:g_transnet}  satisfy
\begin{equation}\label{eq:var_psi}
\mathrm{Var}[\psi_i^{(l)}] \le \delta^{\,l-1}\sigma_0^2,
\qquad i=1,\dots,N_l,\quad l=2,\dots,L.
\end{equation} 
\end{theorem}

\begin{proof}

Combining  \eqref{eq:E_phi}-\eqref{eq:E_psi} with the independence of $\bW_l$ and $\vpsi_{l-1}$, and by noting $\mathbb{E}\left[W_{ij}^{(l)}\right] = \mathbb{E}\left[\psi_{j}^{(l-1)}\right] = 0$,  we get 
\begin{equation}\label{eq:var_phi} 
   \begin{aligned}
       \mathrm{Var}\left[\phi_i^{(l)}\right] &= \mathbb{E}\left[\left(\phi_i^{(l)}\right)^2\right] - \mathbb{E}\left[\phi_i^{(l)}\right]^2 
       = \textstyle\mathbb{E}\left[\left(\sum\limits_{j=1}^{N_{l-1}} W_{ij}^{(l)}\psi_{j}^{(l-1)}\right)^2\right] \\
       &=\textstyle \mathbb{E}\left[\sum\limits_{j=1}^{N_{l-1}} \left(W_{ij}^{(l)}\right)^2 \left(\psi_{j}^{(l-1)}\right)^2\right] 
       = \textstyle\sum\limits_{j=1}^{N_{l-1}} \mathbb{E}\left[\left(W_{ij}^{(l)}\right)^2\right] \mathbb{E}\left[\left(\psi_{j}^{(l-1)}\right)^2\right]\\
       &= \sum_{j=1}^{N_{l-1}}\mathrm{Var}\left[W_{ij}^{(l)}\right]\mathrm{Var}\left[\psi_j^{(l-1)}\right],
   \end{aligned}
   \end{equation}
for  $i=1,\dots,N_l$ and $l=2,\dots,L$.
If each component of $\vpsi_1$ has a variance not more than $\sigma_0^2$, then for the second hidden layer we have from \eqref{eq:sample_W} that
\begin{equation}
\mathrm{Var}\left[\phi_i^{(2)}\right] 
= \textstyle\sum\limits_{j=1}^{N_1}\mathrm{Var}\left[W_{ij}^{(2)}\right]\mathrm{Var}\left[\psi_j^{(1)}\right]
= N_{1} \sigma_2^2 \mathrm{Var}\left[\psi_{j}^{(1)}\right]
        \leq \delta \sigma_0^2.
\end{equation}
Using the fact that $|\tanh(x)| \le |x|$ for any $x\in{\mathbf R}$ and $\mathbb{E}[\vpsi_l]=0$, we obtain 
\begin{equation}\label{eq:var_psi1}
\begin{aligned}
\mathrm{Var}\left[\psi_i^{(2)}\right]
&= \mathbb{E}\left[\left(\psi_i^{(2)}\right)^2\right]
= \mathbb{E}\left[\tanh^2\left(\phi_i^{(2)}\right)\right] \\
&\le \mathbb{E}\left[\left(\phi_i^{(2)}\right)^2\right]
= \mathrm{Var}\left[\phi_i^{(2)}\right] 
\leq \delta \sigma_0^2.
\end{aligned}
\end{equation}
Similarly, for the third hidden layer we have from 
\eqref{eq:var_psi1} that
\begin{equation}
\mathrm{Var}\left[\phi_i^{(3)}\right]
= \textstyle\sum\limits_{j=1}^{N_2}\mathrm{Var}\left[W_{ij}^{(3)}\right]\mathrm{Var}\left[\psi_j^{(2)}\right]
= N_2 \sigma_3^2 \mathrm{Var}\left[\psi_i^{(2)}\right]
\le \delta(\delta \sigma_0^2)= \delta^2 \sigma_0^2
\end{equation}
and consequently
\begin{equation}
\begin{aligned}
\mathrm{Var}\left[\psi_i^{(3)}\right]
&= \mathbb{E}\left[\left(\psi_i^{(2)}\right)^2\right]
\le \mathbb{E}\left[\left(\phi_i^{(2)}\right)^2\right]
= \mathrm{Var}\left[\phi_i^{(2)}\right] 
\leq \delta^2 \sigma_0^2.
\end{aligned}
\end{equation}
By repeating this argument sequentially across all  hidden layers in the same manner, we arrive at the final result \eqref{eq:var_psi} and the proof is completed.
\end{proof}

\subsection{Quantitative study of the neuron behaviors in GTransNet}
For the first hidden layer of the proposed GTransNet \eqref{eq:g_transnet}, the sampling strategy of the direction vector $\mathbf{a}_m$ is identical to that in TransNet; the minor difference lies in how $r_m$ is sampled. In the classic algorithm for TransNet, the selection of the offset parameter $r_m$ is based on the approach \eqref{unif1}, i.e., it is uniformly sampled from the interval $[0,1]$. The corresponding distribution histograms of the hidden-layer neuron activation values were already shown in Figure \ref{fig:density_transnet}, where the distribution is evidently not centrally symmetric. In contrast, GTransNet uses the alternate approach \eqref{unif2}, i.e., sampling $r_m$ from a uniform distribution $\mathcal{U}(-1,1)$. In the following study, as in \eqref{ssp} we again assume the shape parameters of all first hidden-layer neurons take the same value in \eqref{eq:g_transnet}, i.e., $\gamma_m = \gamma$ for $m = 1,2,\ldots,N_1$.

The distribution histograms of the first hidden-layer neuron activation values in GTransNet are shown in the first row of Figure \ref{fig:density_gtransnet}, displaying clear central symmetry, which agrees with the conclusion of equation \eqref{eq:CS_psi} in Lemma \ref{lem:1}. However, it is also easy to find that both sampling schemes suffer from a common drawback -- the distributions exhibit large variance, and most neuron values are concentrated near the saturation regions of the activation function. This phenomenon becomes more pronounced as the shape parameter $\gamma$ increases.

\begin{figure}[htbp]
    \centering
    \includegraphics[width=0.9\linewidth]{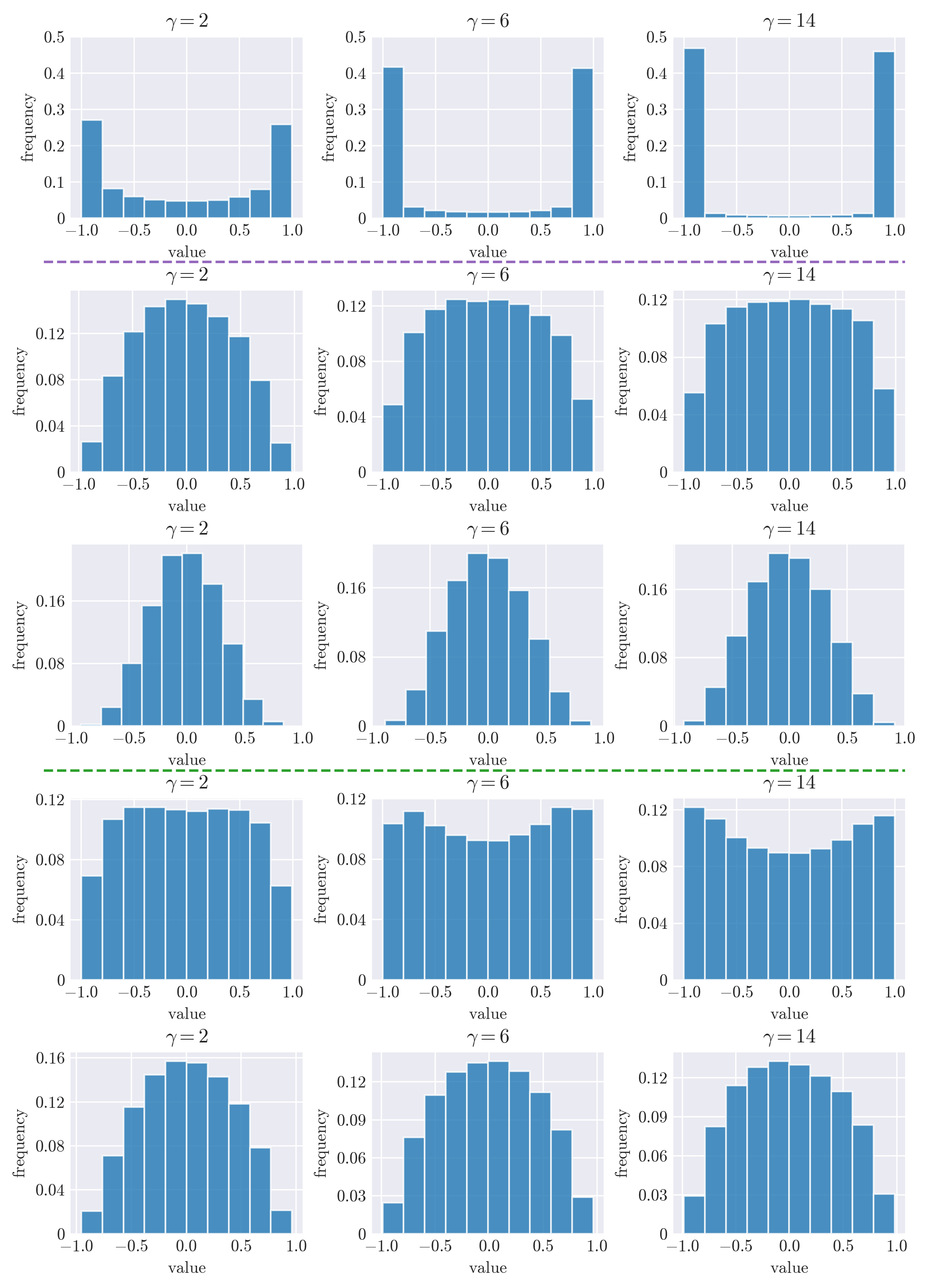}
     \vspace{-0.2cm}
    \caption{Distribution histograms of the hidden-layer neuron activation values in the proposed GTransNet associated with $B_{1.5}(\mathbf{0})$ in two dimensions.  Activation values are collected from 500 input points over the unit square $(-1,1)^2$, with each point corresponding to 2000 different neurons in the first hidden layer and 1000 neurons in the second and third hidden layers. From left to right: $\gamma=2,6, 14.$
    First row: the first hidden layer $\vpsi_1$; Second and third rows: the second and third hidden layers $\vpsi_2$ and $\vpsi_3$ of the GTransNet with $\delta=0.5$;
   Fourth and fifth rows: the second and third hidden layers $\vpsi_2$ and $\vpsi_3$ of the GTransNet with $\delta=0.8$.
}\vspace{-0.3cm}
    \label{fig:density_gtransnet}
\end{figure}

To alleviate this issue, we introduce additional hidden layers into the network and sample its weight matrices according to \eqref{eq:sample_W} within the framework of our GTransNet. As indicated by Theorem \ref{thm:1}, properly choosing the parameter $0<\delta\leq 1$ effectively reduces the variance of neuron values in the second and subsequent hidden layers, preventing them from clustering around the saturation regions. The second to fifth rows of Figure \ref{fig:density_gtransnet} investigate how the response patterns and concentration of the hidden-layer neurons vary with the shape parameter $\gamma$ and the network depth: the distribution histograms of the second and third hidden-layer neuron activation values gradually approach a Gaussian distribution and more frequently fall within the non-saturated region, thereby significantly enhancing the model’s expressive capability.
Furthermore, as $\delta$ decreases and the network depth increases, the distributions become more concentrated near zero. This observation suggests that when $\delta$ is small, the GTransNet architecture should not be designed with too many hidden layers, since this would lead to neuron values being overly concentrated near zero.

Moreover, the first two rows of Figure \ref{fig:GTransNet_basis} visualize some sample neurons from the second hidden layer in the proposed GTransNet. The first row corresponds to sampling $\mathbf{W}_2$ directly from a standard Gaussian distribution $\mathcal{N}(0,1)$, while the second row corresponds to our variance-controlled sampling strategy \eqref{eq:sample_W} with $\delta = 0.5$. The results demonstrate that general Gaussian sampling could cause neurons to nearly collapse into saturation zones under large $\gamma$, whereas our proposed sampling strategy \eqref{eq:sample_W} can effectively avoid this issue.
As $\gamma$ increases, the second hidden-layer neurons of GTransNet gradually exhibit distinct  high-frequency and multiscale features uniformly distributed across the whole domain, which is particularly beneficial for handling problems with high-frequency solutions. The third row of Figure \ref{fig:GTransNet_basis} shows some sample neurons from the third hidden layer produced using the sampling strategy \eqref{eq:sample_W}, and we again observe similar but more pronounced neural features compared with those in the second hidden layer.

\begin{figure}[!t]
    \centering
    \includegraphics[width=0.9\linewidth]{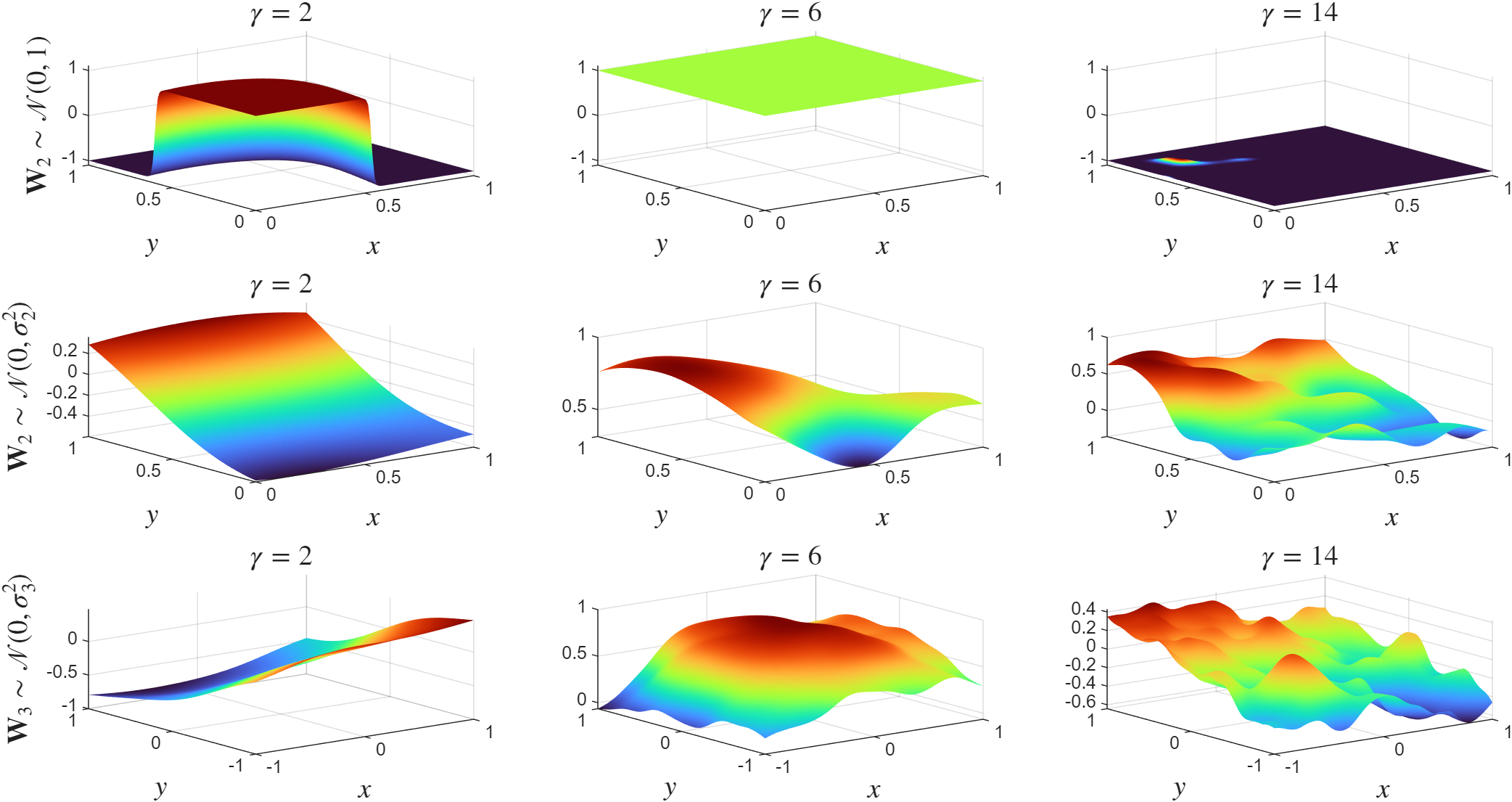}
    \caption{Visualization of some typical second and third hidden-layer neurons, $\psi^{(2)}_m(\bx)$ and $\psi^{(3)}_m(\bx)$, in the proposed GTransNet. From left to right: $\gamma=2,6, 14.$
   Top row: $\psi^{(2)}_m(\bx)$ with $\mathbf{W}_2\sim\mathcal{N}(0,1)$; Middle row: $\psi^{(2)}_m(\bx)$ with $\mathbf{W}_2\sim\mathcal{N}(0,\sigma_2^2)$ and $\delta = 0.5$; Bottom row: $\psi^{(3)}_m(\bx)$ with $\mathbf{W}_3\sim\mathcal{N}(0,\sigma_3^2)$ and $\delta = 0.5$.}\vspace{-0.3cm}
    \label{fig:GTransNet_basis}
\end{figure}

In summary, simply increasing the value of the shape parameter $\gamma$ in TransNet causes the non-saturated region to narrow rapidly; as a result, a large number of hidden-layer neurons is required to adequately capture high-frequency features -- an approach that becomes infeasible in high-dimensional settings. In contrast, our proposed GTransNet retains the feature extraction mechanism of TransNet while introducing additional hidden layers with a well-designed weight sampling strategy, which effectively controls the variances and regulates the neuron distribution.
Especially under large $\gamma$, the additional layers can produce  neurons with desired features that are nicely distributed within the non-saturated region. These layers therefore serve as feature enhancement layers, substantially improving both the generalization and expressive power of the neural network model.

\begin{remark}
It is worth noting that although the proposed GTransNet framework allows the construction of deeper architectures, experimental results (see Subsection 4.1) show that using three or more hidden layers may not provide additional accuracy gains compared with two hidden layers; instead, it could substantially increase the overall computational cost due to the need to evaluate more derivatives when assembling the target least-squares system. On the other hand, compared with TransNet, introducing one additional hidden layer and incorporating the proposed weight–presetting strategy enables more effective extraction of  high-frequency and multiscale features. Hence, a two-hidden-layer GTransNet achieves a superior balance between accuracy and efficiency, making it a more reasonable and effective choice.
\end{remark}

\section{Numerical Experiments}\label{sec:4}
In this section, we numerically demonstrate the accuracy, efficiency, and robustness of the proposed GTransNet \eqref{eq:g_transnet} on solving the steady-state PDE problems \eqref{modelp}, especially the ones with highly oscillatory solutions. The problems cover five different two-dimensional computational domains and two three-dimensional domain:
\begin{itemize}
\item A 2D unit square domain $\Omega_1 = (0,1)\times (0,1)$.
\item A 2D shifted unit square domain $\Omega_2 = (4,5)\times (7,8)$.
\item A 2D kite-shape domain $\Omega_3$ with  its boundary
$\partial\Omega_3
= \{(x,y)\,|\, x = 0.5\cos \theta + 0.3\cos(2\theta) - 0.2, ~y = 0.6\sin \theta,~ 0\leq\theta<2\pi\}.$
\item A 2D flower-shape domain $\Omega_4$ with its boundary 
$\partial\Omega_4
= \{(x,y)\,|\, x = (0.5 - 0.1\cos(6\theta))\cos \theta, ~y = (0.5 - 0.1\cos(6\theta))\sin \theta,~ 0\leq\theta<2\pi\}.$
\item A 2D L-shaped domain $\Omega_5 = (0,1)\times (0,1) \backslash [0.5,1] \times [0.5,1].$
\item A 3D ball domain with radius 0.5 centered at the origin 
$\Omega_6 = \{(x,y,z)\mid x^2 + y^2 + z^2 < 0.5^2\}.$ 
\item A 3D box domain $\Omega_7=(0,1)\times (0,1)\times (0,1)$.
\end{itemize}
To ensure complete coverage of each target domain, the ball $B_R(\mathbf{x_c})$ used in TransNet and GTransNet are centered at $(0.5,0.5)$, $(5.5,6.5)$, $(0,0)$, $(0,0)$, $(0.5,0.5)$, $(0,0,0)$, and $(0.5,0.5,0.5)$ for $\Omega_1$ through $\Omega_7$, respectively, with a fixed radius $R=0.8$ maintained across all cases.
For simplicity, we also let the shape parameters of all first hidden-layer neurons take the same value $\gamma$.

To evaluate the numerical accuracy, we introduce the relative \(L^2\) error:
\[
\frac{\|u - u_{\mathrm{NN}}\|_{L^2}}{\|u\|_{L^2}},
\]
where \(u_{\mathrm{NN}}\) denotes the numerical solution obtained using the GTransNet or TransNet method, and \(u\) is the exact solution. 
The numerical errors are computed with  \(K_{\mathrm{test}} = 10^4\) test points. The number of neurons in the single hidden layer of TransNet \eqref{snn} is set to be equal to that of the last hidden layer in GTransNet, i.e., \(M=N_L:=N\), so that their resulting least squares systems  are of the same size in  comparison. The variance control parameter in GTransNet is chosen as  $\delta = 0.5$ in \eqref{eq:sample_W}.

All experiments are conducted on a Windows system equipped with an Nvidia RTX 2080 Ti GPU. PyTorch’s \verb|torch.linalg.lstsq| API is used to solve the least-squares problems with double-precision floating-point arithmetic. All error results are obtained by repeating each experiment 20 times.

\subsection{Examples with smooth solutions}\label{sexp}

We evaluate the efficiency and accuracy of GTransNet by comparing it with TransNet on the following three problems with smooth solutions. Dirichlet boundary conditions are imposed for all these problems.
\begin{itemize}
\item[]
\textbf{(S1).} {The Poisson equation} on the 2D shifted unit square domain $\Omega_2$:
 \begin{equation}\label{eq:Poisson}
     -\Delta u  = f, \quad (x,y)\in \Omega_2
 \end{equation}
with  the exact solution  
\begin{equation}
 u = \sin(2\pi x) \sin(2\pi y) + e^{-x - y}.
\end{equation}
\item[] 
\textbf{(S2).} {The Poisson equation} on the 3D ball domain $\Omega_6$ with the exact solution 
\begin{equation}
u = \sin(\pi x) \sin(\pi y) \sin(\pi z).\end{equation}
\item[] 
\textbf{(S3).} {The steady-state Navier-Stokes equations} on the 2D unit square domain $\Omega_1$:
\begin{equation}\label{eq:NS}
\begin{cases}
    -\nu\Delta \mathbf{u} + (\mathbf{u}\cdot\nabla)\mathbf{u} + \nabla p = \mathbf{f}, & \; (x,y)\in \Omega_1, \\
    \nabla \cdot \mathbf{u} = 0, & \; (x,y) \in \Omega_1,
\end{cases}
\end{equation}
where \( \mathbf{u} = (u, v) \) denotes the velocity field and $p$ the pressure, the kinetic viscocity \(\ds \nu = 0.1\), and the exact solution is given by 
\begin{equation}
u = e^x\sin(\pi y), ~~
v = e^x\cos(\pi y),~~
p = \sin(0.5\pi x)\cos(0.5\pi y).
\end{equation}
\end{itemize}
For all the problems mentioned above, the  Dirichlet boundary conditions and the source terms are determined accordingly.

We compare the performance of the TransNet and GTransNet methods. In particular, GTransNets with two and three hidden layers (i.e., $L = 2$ and $L = 3$) are both tested to investigate the effect of network depth. The numbers of neurons in the second and third hidden layers of the GTransNets are kept identical if $L=3$, and the number of neurons in the first hidden layer of the GTransNets is fixed at $N_1 = 800$ for the 2D Poisson problem (S1) and the 2D steady-state Navier-Stokes problem (S3), and $N_1 = 3000$ for the 3D Poisson problem (S2).
The numbers of interior and boundary collocation points for training, $(K_{\mathrm{int}}, K_{\mathrm{bdy}})$, are set to $(900, 200)$ for (S1) and (S3), and $(2512, 532)$ for (S2). The shape parameter $\gamma$ is set to be $1$ for (S2) and (S3), and $2$ for (S1). Note that (S3) is a nonlinear problem, so the Picard iteration method is further employed to solve the resulting nonlinear least-squares system, where the iterative scheme is obtained by linearizing the nonlinear convection term as follows:
\begin{equation}
\begin{cases}
    -\nu\Delta \mathbf{u}^{(k+1)} + \nabla p^{(k+1)} + (\mathbf{u}^{(k)}\cdot\nabla)\mathbf{u}^{(k+1)} = \mathbf{f}, & \; (x,y)\in \Omega_1, \\
    \nabla \cdot \mathbf{u}^{(k+1)} = 0, & \; (x,y) \in \Omega_1,
\end{cases}
\end{equation}
for $k\ge 1$ with the initial value $\mathbf{u}^{(0)} = (1,1)$, and the iteration terminates if $\|\mathbf{u}^{(k+1)}-\mathbf{u}^{(k)}\|_{\infty}$ is less than $10^{-12}$.

\begin{figure}[htbp]
    \centering
    \includegraphics[width=0.65\linewidth]{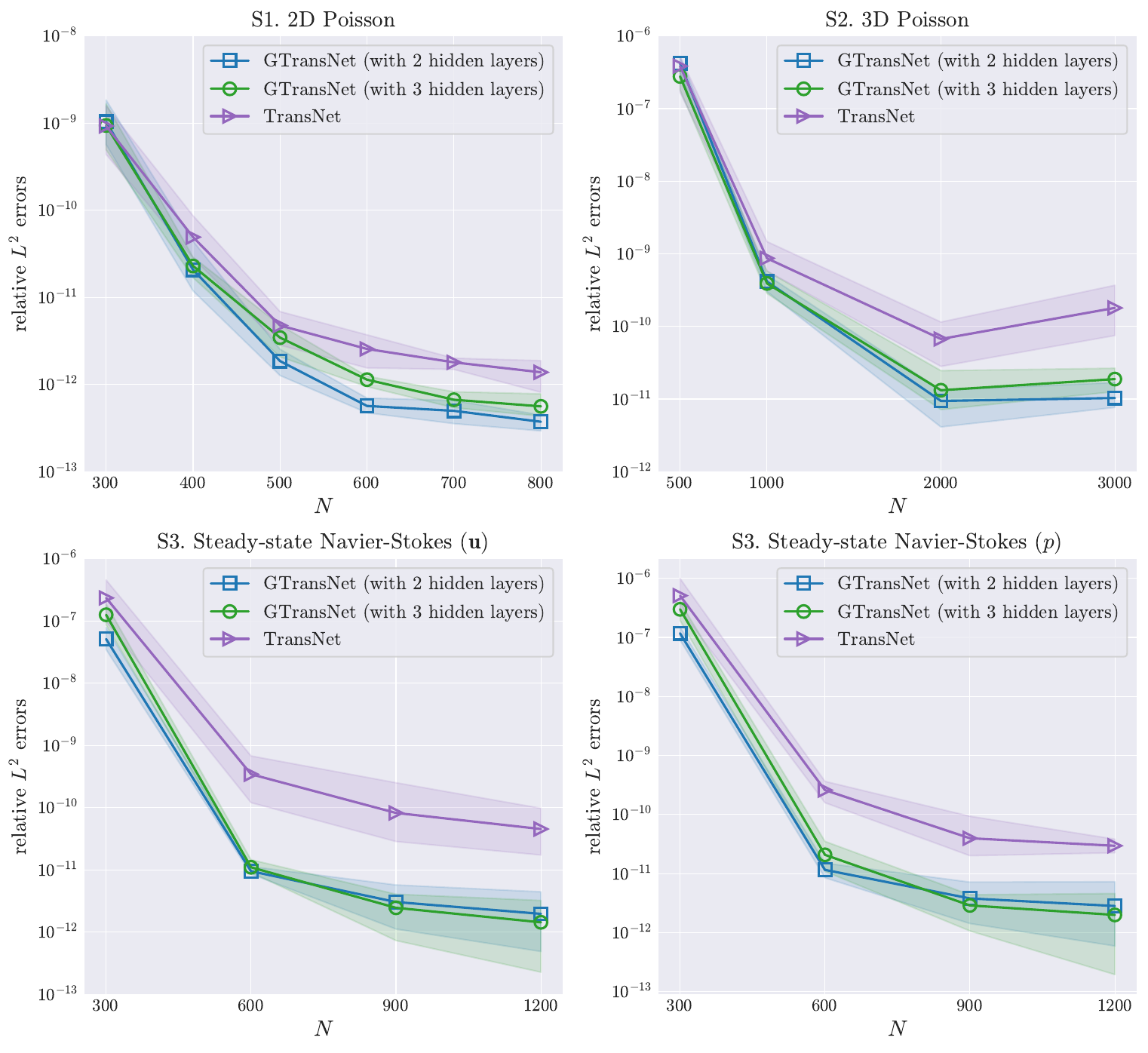}
     \vspace{-0.2cm}
    \caption{Plots of relative \( L^2 \) errors of numerical solutions produced by the TransNet method and the GTransNet method with two and three hidden layers in Section \ref{sexp},  where $N$ denotes the number of last hidden-layer neurons. Top-left: (S1); Top-right: (S2); Bottom row: (S3).}\vspace{-0.3cm}
    \label{fig:smooth_error}
\end{figure}

\begin{figure}
    \centering
    \includegraphics[width=0.9\linewidth]{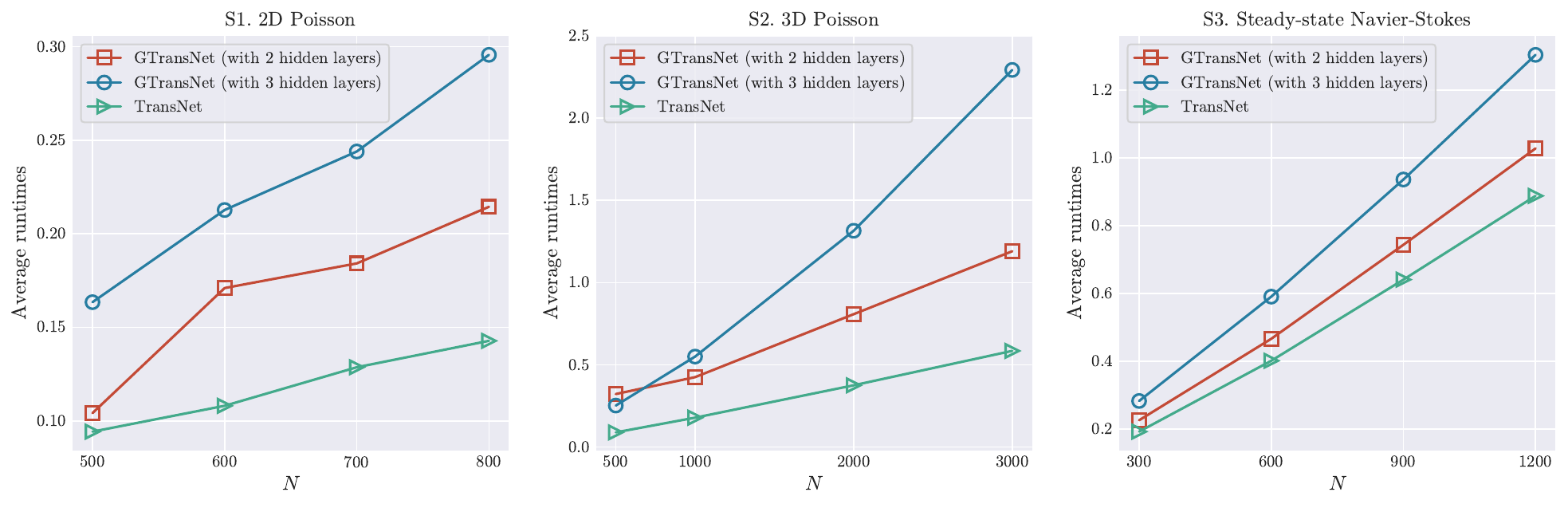}
     \vspace{-0.2cm}
    \caption{Average running times (in seconds) of the TransNet method and the GTransNet method with two and three hidden layers in Section \ref{sexp}, where $N$ denotes the number of neurons in the last hidden layer. From left to right: (S1), (S2) and (S3).}\vspace{-0.3cm}
    \label{fig:runtime}
\end{figure}

In Figure~\ref{fig:smooth_error}, we plot the relative $L^2$ errors of all tested TransNets and GTransNets, highlighting the effect of the number of neurons $N$ in the last hidden layer. The corresponding running times of these tests are reported in Figure~\ref{fig:runtime}. The following notable trends can be observed from these results:
(i) All models achieve high accuracy across all three test problems, confirming the reliability of the underlying network architectures.
(ii) The running time increases almost linearly with respect to the value of $N$ for both TransNet and GTransNet across all three problems.
(iii) For any fixed $N$, the errors of GTransNet are consistently and significantly smaller than those of TransNet. Conversely, to reach a comparable error level, GTransNet requires far fewer neurons in the last hidden layer, resulting in a much smaller least-squares system.
(iv) Although the resulting least-squares systems are of the same size, GTransNet incurs higher running times than TransNet, since the addition of extra hidden layers increases the computational cost of assembling the least-squares system and solving it.
(v) Adding more hidden layers in GTransNet, such as increasing from $L = 2$ to $L = 3$, only slightly improves solution accuracy but inevitably increases computational time.
Hence, the GTransNet with two hidden layers is used in all subsequent experiments due to its favorable balance between accuracy and efficiency, and is hereafter referred to simply as GTransNet.

\subsection{Examples with highly oscillatory solutions}
We now demonstrate the performance of the GTransNet in solving PDEs with highly oscillatory solutions, including the Poisson equation, the Helmholtz equation, the multiscale elliptic problem, and the steady-state Allen–Cahn equation. All aforementioned 2D and 3D computational domains except $\Omega_2$ are involved, with the numbers of interior and boundary collocation points for training, $(K_{\mathrm{int}}, K_{\mathrm{bdy}})$, being $(10000, 600)$, $(9256, 500)$, $(10752, 500)$, $(10800, 580)$, $(12712, 2500)$ and $(16000,4800)$ for $\Omega_1$ and $\Omega_3$ to $\Omega_7$, respectively.
Unless otherwise specified, the number of neurons in the first hidden layer of all tested GTransNets is fixed at $N_1 = 6000$ for the 2D problems and $N_1 = 12000$ for the 3D problems in order to capture the intrinsic features.

\subsubsection{The Poisson equation}
Let us first investigate the effectiveness of the GTransNet method  in solving the Poisson equation 
for high-frequency solutions. We consider the following three representative cases:

\begin{itemize}
\item[] \textbf{Case 1.} The 2D Kite-shaped domain $\Omega_3$ with the exact solution
\begin{equation}
u(x,y) = \sin(2\pi x) \sin(30\pi y).
\end{equation}
\item[] \textbf{Case 2.} The 2D L-shaped domain $\Omega_5$ with the exact solution 
\begin{equation}
u(x,y) = \cos(100 x) \cos(100 y).
\end{equation}
\item[] \textbf{Case 3.} The 3D ball domain $\Omega_6$ and the exact solution
\begin{equation}
u(x,y,z) = \sin(2x) + \cos(10 y) + \sin(36 z).
\end{equation}
\end{itemize}
The source term s and the Dirichlet boundary condition are again determined accordingly for each case.

\begin{figure}[ht!]
    \centering
    \includegraphics[width=0.9\linewidth]{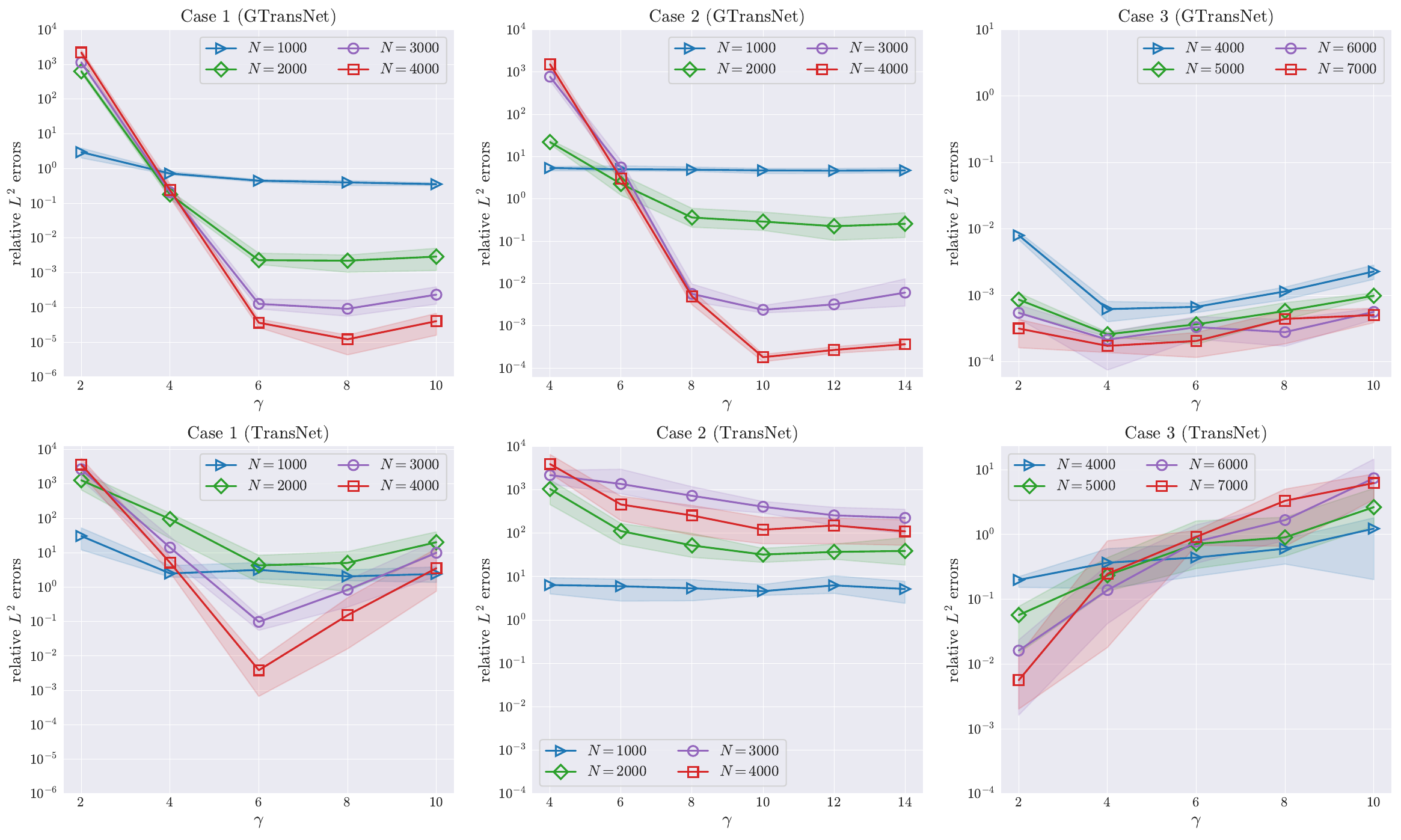}
     \vspace{-0.2cm}
    \caption{Plots of the relative $L^2$ errors of numerical solutions produced by the GTransNet method
    ({top row}) and the TransNet method ({bottom row}) with different values of the last hidden-layer neurons $N$ and the shape parameter $\gamma$ for the three cases of the Poisson equation 
    in Section 4.2.1. From left to right: Cases 1 (2D), Case 2 (2D), and Case 3 (3D).
    }\vspace{-0.3cm}
    \label{fig:Poisson_convergence}
\end{figure}

\begin{figure}[ht!]
    \centering
    \includegraphics[width=0.9\linewidth]{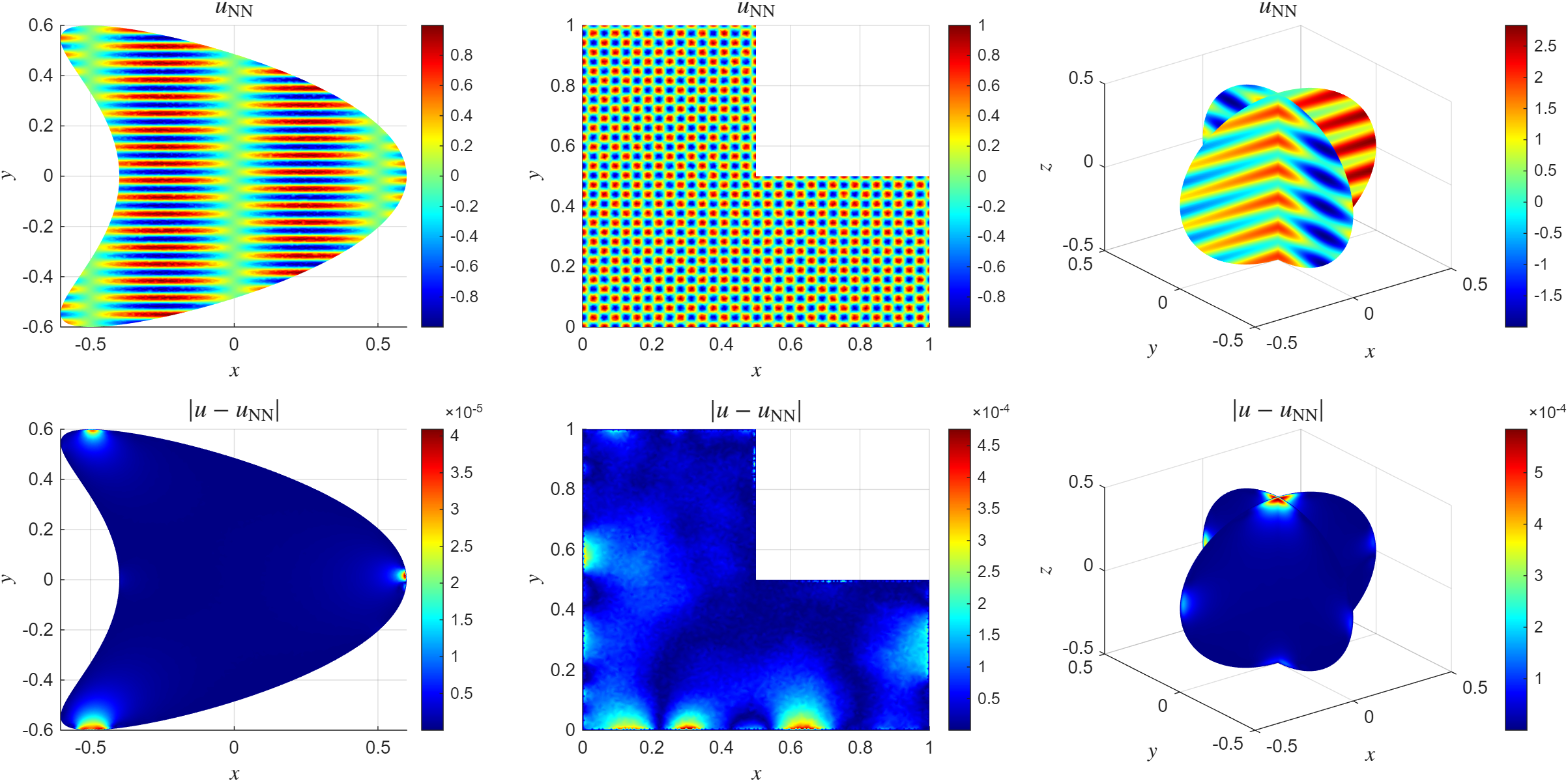}
     \vspace{-0.1cm}
    \caption{Numerical solutions ({top row}) and the corresponding absolute point-wise errors ({bottom row}) produced by the GTransNet method for the three cases of the Poisson equation 
    in Section 4.2.1. Left column: Case 1 (2D), $N = 4000, \gamma = 8$; Middle column: Case 2 (2D), $N = 4000, \gamma = 10$; Right column: Case 3 (3D), $N = 7000, \gamma = 4$, at the two cross sections $x = 0$ and $y = 0$.
    }\vspace{-0.3cm}
    \label{fig:Poisson_equation}
\end{figure}
 
Figure \ref{fig:Poisson_convergence} illustrates the impact of the shape parameter $\gamma$ on the relative $L^2$ error of GTransNet and TransNet in solving the Poisson equation, with different numbers of neurons $N$ in the last hidden layer. The results indicate that:
(i) When $\gamma$ is small, both methods exhibit low accuracy across all three test cases, as the relatively smooth neurons are unable to effectively capture high-frequency features.
(ii) As $\gamma$ increases, GTransNet significantly outperforms TransNet, achieving better accuracy by several orders of magnitude. In Case 1, TransNet attains a relative $L^2$ error of approximately $\mathcal{O}(10^{-3})$, whereas in Case 2 and Case 3, it almost fails completely.
(iii) For larger values of $\gamma$, increasing the number of neurons in the last hidden layer of GTransNet further reduces the error.
(iv) For all values of $N$, the error may eventually reach a ``saturation'' level as $\gamma$ increases. This phenomenon can be attributed to two factors: when $N$ is small, the model’s expressive capacity is limited, leading to rapid saturation; when $N$ is large, the model has sufficient expressive capacity, and saturation occurs later, possibly due to the fixed number of collocation points used here.

Figure \ref{fig:Poisson_equation} presents several numerical results of GTransNet, including the predicted solutions and corresponding absolute pointwise errors. For Case 1, Case 2, and Case 3, the shape parameters $\gamma$ are set to $8$, $10$, and $4$, respectively, and the numbers of neurons in the last hidden layer $N$ are $4000$, $4000$, and $7000$, respectively. We observe that GTransNet achieves excellent numerical accuracy in high-frequency scenarios, with the absolute pointwise errors consistently maintained at low levels, approximately $\mathcal{O}(10^{-5})$ to $\mathcal{O}(10^{-4})$ and $\mathcal{O}(10^{-4})$, respectively.

\subsubsection{The Helmholtz equation}
We next consider the Helmholtz equation in the 2D flower-shape domain $\Omega_4$:
\begin{equation}\label{eq:Helmholtz}
        -\Delta u - k^2 u = 0, \quad (x, y) \in \Omega_4, 
\end{equation}
where the wave number is defined by
$
\textstyle k = \frac{2\pi \nu}{c},
$
with wave speed $c = 340\,\mathrm{m/s}$ and frequency $\nu$ taking three different values, 4000 Hz, 6000 Hz, and 8000 Hz, respectively. The corresponding exact solution is given by
\begin{equation}
    u(x, y) = \textstyle\sin\left(\frac{k}{\sqrt{2}}x\right)\sin\left(\frac{k}{\sqrt{2}}y\right),
\end{equation}
which is highly oscillatory and the Dirichlet boundary condition is imposed.

\begin{figure}[ht!]
    \centering
    \includegraphics[width=0.9\linewidth]{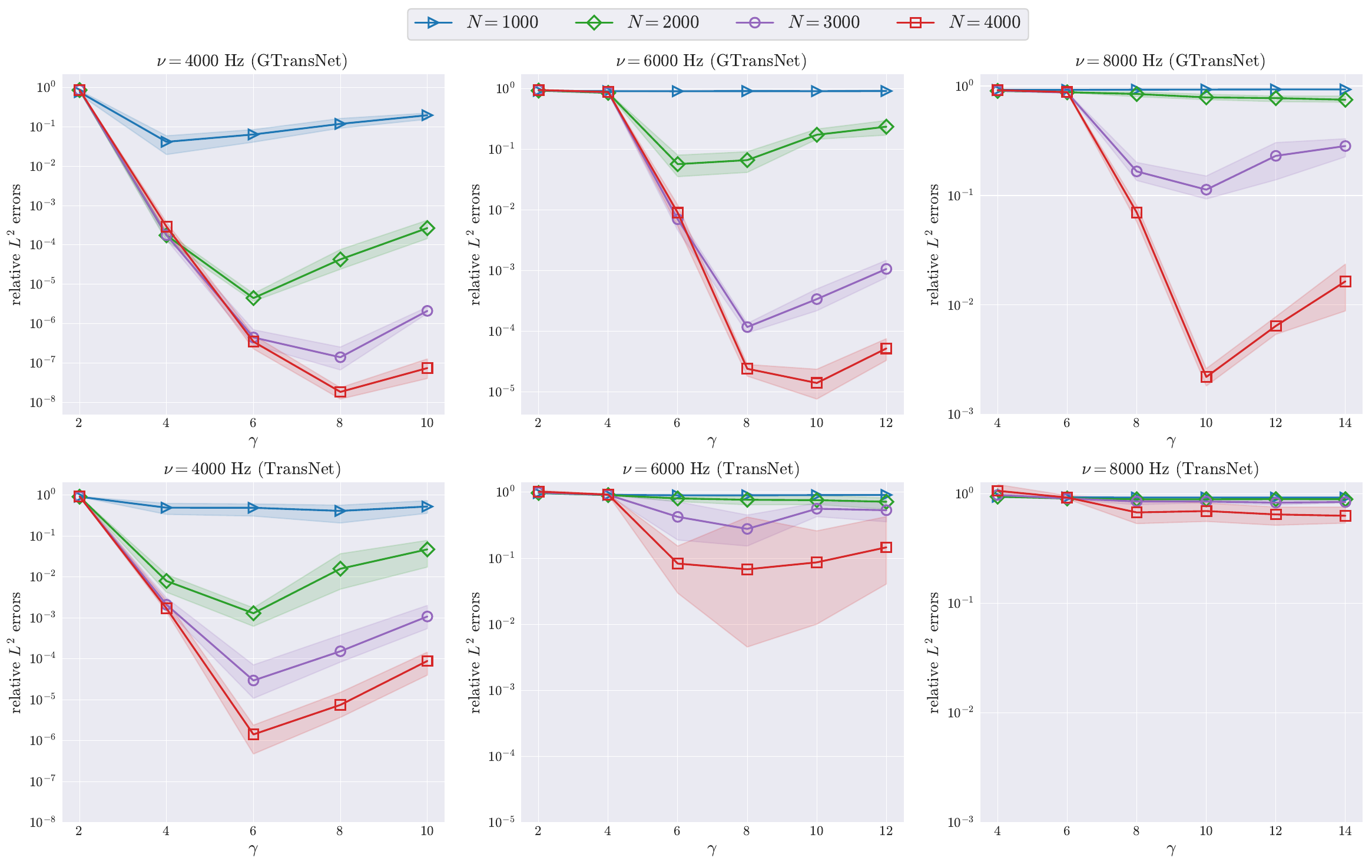}
     \vspace{-0.2cm}
    \caption{Plots of the relative $L^2$ errors of numerical solutions produced by the GTransNet method ({top row}) and the TransNet method ({bottom row}) with different values of the last hidden-layer neurons $N$ and the shape parameter $\gamma$ for the 2D Helmholtz equation in  Section 4.2.2. From left to right: $\nu = 4000$, 6000, and 8000 Hz.}\vspace{-0.3cm}
    \label{fig:Helmholtz_convergence}
\end{figure}

\begin{figure}[ht!]
    \centering
    \includegraphics[width=0.9\linewidth]{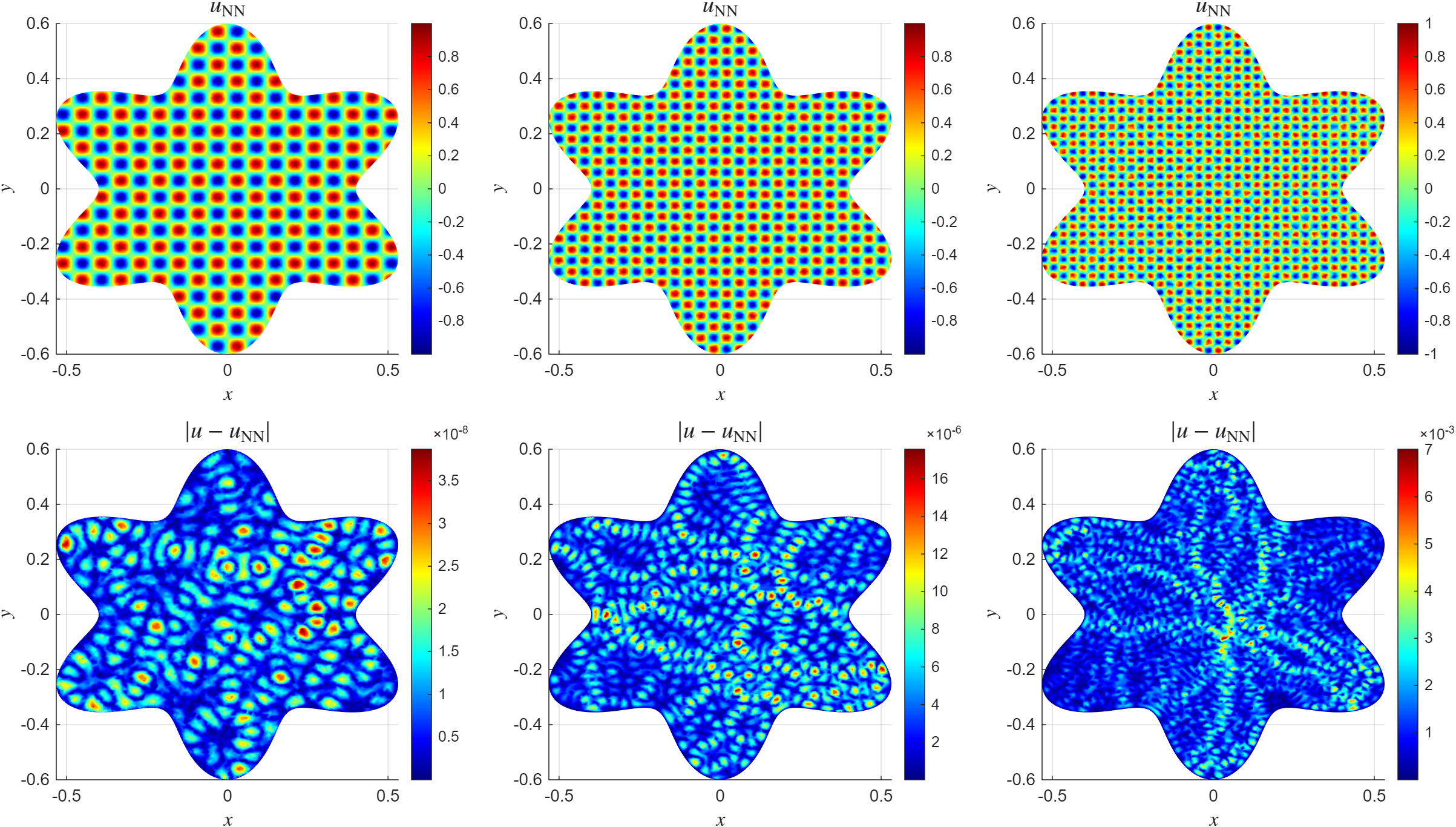}
     \vspace{-0.1cm}
    \caption{Numerical solutions ({top row}) and the corresponding absolute point-wise errors ({bottom row}) produced by the GTransNet method for the 2D Helmholtz equation  in Section 4.2.2. Left column: $\nu = 4000$ Hz, $N=4000, \gamma = 6$; middle column: $\nu = 6000$ Hz, $N=4000, \gamma = 10$; right column: $\nu = 8000$ Hz, $N=4000,\gamma = 10$.}\vspace{-0.3cm}
    \label{fig:Helmholtz_equation}
\end{figure}

Figure \ref{fig:Helmholtz_convergence} illustrates the effect of the shape parameter $\gamma$ on the relative $L^2$ error of GTransNet and TransNet for solving the Helmholtz equation \eqref{eq:Helmholtz}, while also comparing performance under varying numbers of neurons $N$ in the last hidden layer. The results exhibit behaviors similar to those observed for the Poisson equation in Section 4.2.1. TransNet performs reasonably well at the low frequency ($\nu = 4000$ Hz), but its accuracy deteriorates markedly at higher frequencies ($\nu = 6000$ Hz and $8000$ Hz). In contrast, GTransNet maintains much lower errors and is less sensitive to sampling results of the hidden layer neural parameters across all three frequencies, demonstrating its robustness and superior capability in capturing high-frequency features.

Figure \ref{fig:Helmholtz_equation} presents several numerical results of GTransNet, including the numerical solutions and the corresponding absolute pointwise errors. For the frequencies $\nu = 4000$ Hz, $6000$ Hz, and $8000$ Hz, the shape parameter $\gamma$ is set to $6$, $10$, and $10$, respectively, while the number of neurons in the last hidden layer $N$ is fixed at $4000$. It is observed that GTransNet achieves excellent numerical accuracy under all three frequency conditions, with absolute pointwise errors of approximately $\mathcal{O}(10^{-8})$, $\mathcal{O}(10^{-5})$, and $\mathcal{O}(10^{-3})$, respectively. Although the error naturally increases with frequency, these results demonstrate that GTransNet remains both stable and accurate even in high-frequency regimes.

\subsubsection{The multiscale elliptic problem}
We consider a  multiscale elliptic PDE problem defined on a 2D square domain $\Omega_1$:
\begin{equation}\label{eq:multiscale}
    -\nabla \cdot \left( A^{\varepsilon}(x,y) \nabla u \right) = f, \quad (x,y) \in \Omega_1, 
\end{equation}
where the diffusion coefficient exhibits rapid oscillations and is defined as:
\begin{equation}
A^{\varepsilon}(x,y) = \frac{1}{4 + \cos\left(\frac{2\pi (x^2 + y^2)}{\varepsilon}\right)},
\end{equation}
with \(\varepsilon\) representing the scale separation parameter. The source term is given by
$f(x,y) = - (x^2 + y^2)$.
The corresponding analytical solution reads
\begin{equation}
\textstyle
u(x,y) = u_p(x,y) + \frac{\varepsilon}{16 \pi}(x^2 + y^2) \sin\left(\frac{2\pi (x^2 + y^2)}{\varepsilon}\right) + \frac{\varepsilon^2}{32 \pi^2} \cos\left(\frac{2\pi (x^2 + y^2)}{\varepsilon}\right),
\end{equation}
where $u_p(x,y) = \frac14(x^2+y^2)^2$. The Dirichlet boundary condition is again  imposed.

Such multiscale solution problems are common in engineering applications, including materials science (e.g., microscale diffusion in composite materials), groundwater flow simulation, and electromagnetic wave propagation. As $\varepsilon$ decreases, the diffusion coefficient exhibits high-frequency oscillations, resulting in solutions with complex high-frequency and multiscale structures. Traditional numerical methods, such as finite element or finite difference methods, usually require extremely fine meshes to resolve these rapid variations, leading to high computational costs and reduced efficiency. This example serves to assess the stability and accuracy of GTransNet in the presence of highly oscillatory diffusion coefficients.
We compare the performance of GTransNet with TransNet under three different values of the scale separation parameter $\varepsilon$, namely $0.5$, $0.2$, and $0.1$, respectively, highlighting GTransNet’s effectiveness in capturing multiscale features.

\begin{figure}[ht!]
    \centering
    \includegraphics[width=0.9\linewidth]{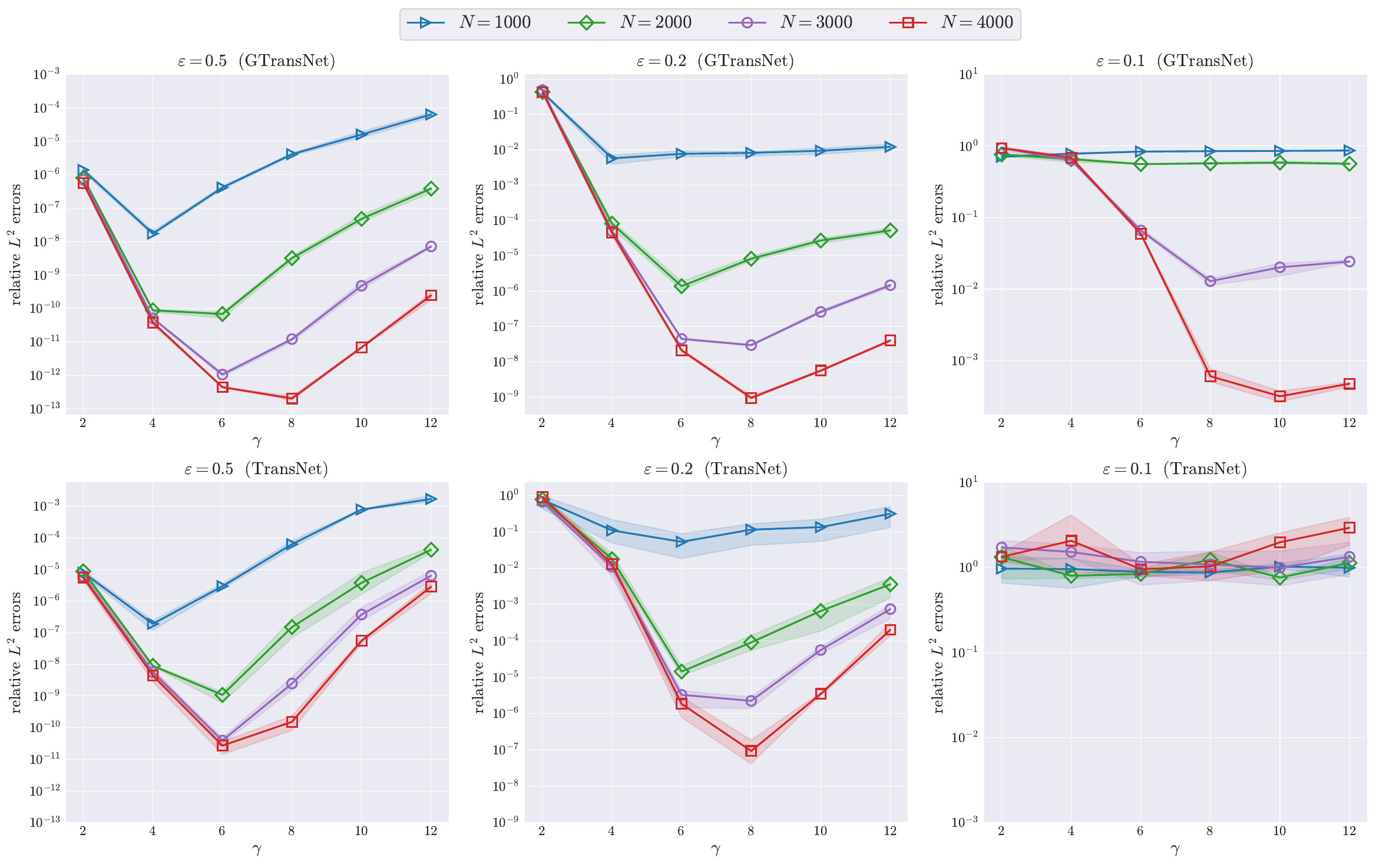}
     \vspace{-0.2cm}
    \caption{Plots of the relative $L^2$ errors of numerical solutions produced by the GTransNet method ({top row}) and the TransNet method ({bottom row}) with different values  of the last hidden-layer neurons $N$ and the shape parameter $\gamma$ for the 2D multiscale elliptic problem in Section 4.2.3. From left to right: $\varepsilon = 0.5$, 0.2, and 0.1.}\vspace{-0.3cm}
    \label{fig:multiscale_convergence}
\end{figure}

\begin{figure}[ht!]
    \centering
    \includegraphics[width=0.9\linewidth]{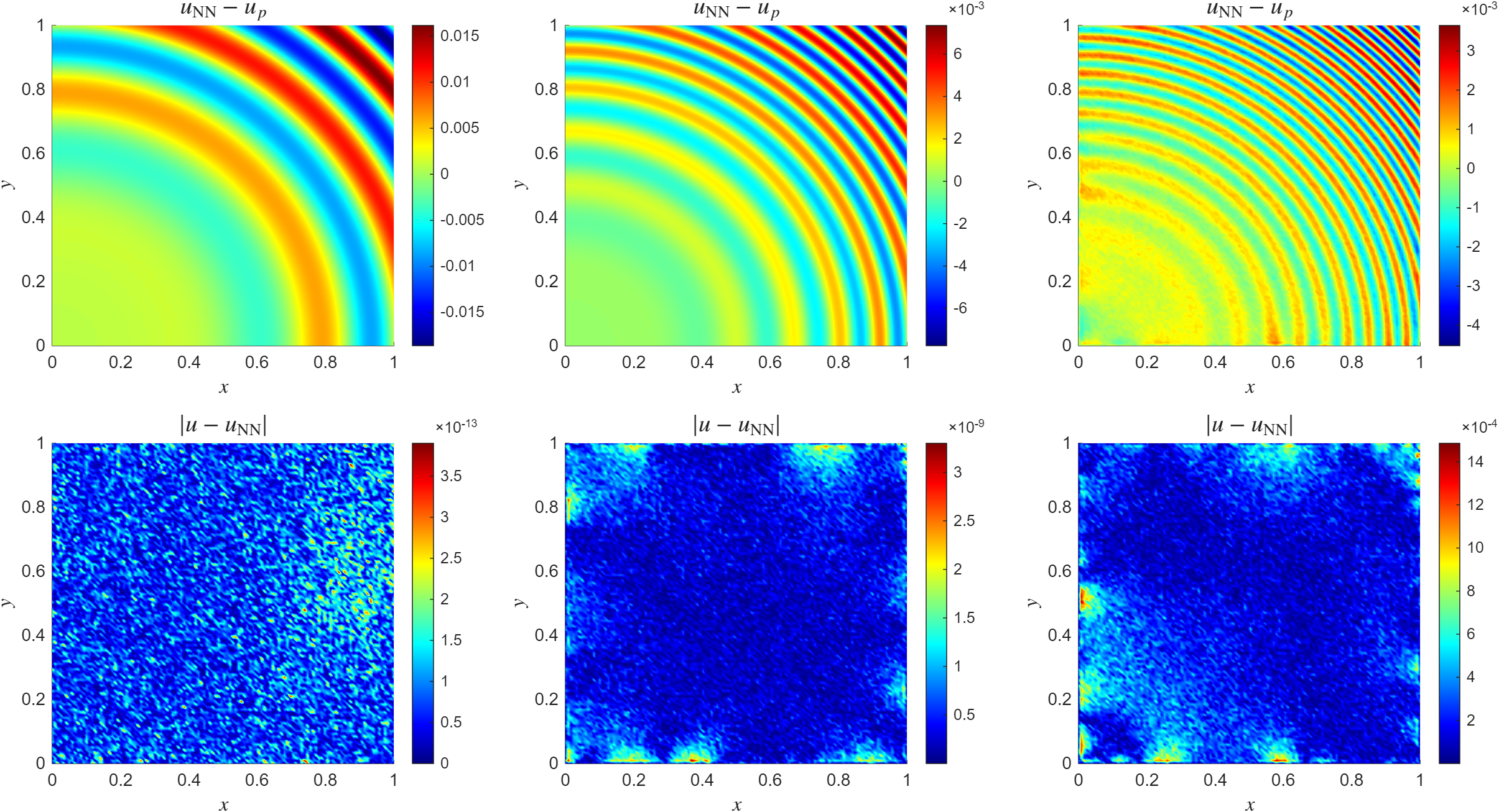}
    \vspace{-0.2cm}
    \caption{Numerical solutions $u_{\mathrm{NN}} - u_p$ ({top row}) and the corresponding absolute point-wise errors of numerical solutions ({bottom row}) produced by the GTransNet method for the 2D multiscale elliptic problem \eqref{eq:multiscale}  in Section 4.2.3. Left column: $\varepsilon = 0.5, N = 4000, \gamma = 6$; Middle column: $\varepsilon = 0.2, N = 4000, \gamma = 8$; Right column: $\varepsilon = 0.1, N = 4000, \gamma = 8$.}\vspace{-0.3cm}
    \label{fig:multiscale}
\end{figure}

Figure \ref{fig:multiscale_convergence} illustrates the effect of the shape parameter $\gamma$ on the relative $L^2$ error of GTransNet and TransNet for solving the multiscale elliptic equation \eqref{eq:multiscale}, while also exploring the effect of the number of neurons $N$ in the last hidden layer. The results exhibit behaviors similar to those observed for the Poisson and Helmholtz equations in the previous subsections. As $\varepsilon$ decreases, the solution accuracy of TransNet deteriorates significantly across all configurations. In contrast, GTransNet consistently performs much better than TransNet and maintains much lower errors for all three scale separation parameters, clearly demonstrating its robustness and strong capability in handling multiscale elliptic problems with highly oscillatory coefficients and solutions.

Figure \ref{fig:multiscale} presents several numerical results of GTransNet, including $u_{\mathrm{NN}} - u_p$ and the absolute pointwise errors of the numerical solutions. The last hidden layer contains $N = 4000$ neurons, with $\gamma$ set to $6$, $8$, and $8$ for $\varepsilon = 0.5$, $0.2$, and $0.1$, respectively. The results show that GTransNet performs excellently under all three parameter settings, with the absolute pointwise errors being approximately $\mathcal{O}(10^{-13})$, $\mathcal{O}(10^{-9})$, and $\mathcal{O}(10^{-4})$, respectively.

\subsubsection{The steady-state Allen-Cahn equation}
In this example, we consider the steady-state Allen-Cahn equation on the  3D box domain $\Omega_7$:
\begin{equation}\label{eq:allen_cahn}
    -\Delta u + \textstyle\frac{1}{0.05^2}(u^3 - u) = f, \quad (x,y,z) \in \Omega_7
\end{equation}
with the periodic boundary condition. The exact solution is given by
\begin{equation}
u(x,y,z) = \cos(36x(1-x))\cos(36y(1-y))\cos(36z(1-z))
\end{equation}
and then the source term $f$ is determined accordingly.
Note that this equation is nonlinear thus the Picard iteration method is again used to solve the resulting nonlinear least-squares system.
The iterative scheme is  obtained by treating the nonlinear reaction term explicitly with stabilization as follows:
\begin{equation}
\textstyle -\Delta u^{(k+1)} + \frac{S}{0.05^2} u^{(k+1)} =
f - \frac{1}{0.05^2}((u^{(k)})^3-(1+S)u^{(k)})
\end{equation}
for $k\ge 1$, with the stabilizing parameter $S=2$  and the initial value $u^{(0)} = 1$.

\begin{figure}[ht!]
    \centering
    \includegraphics[width=0.65\linewidth]{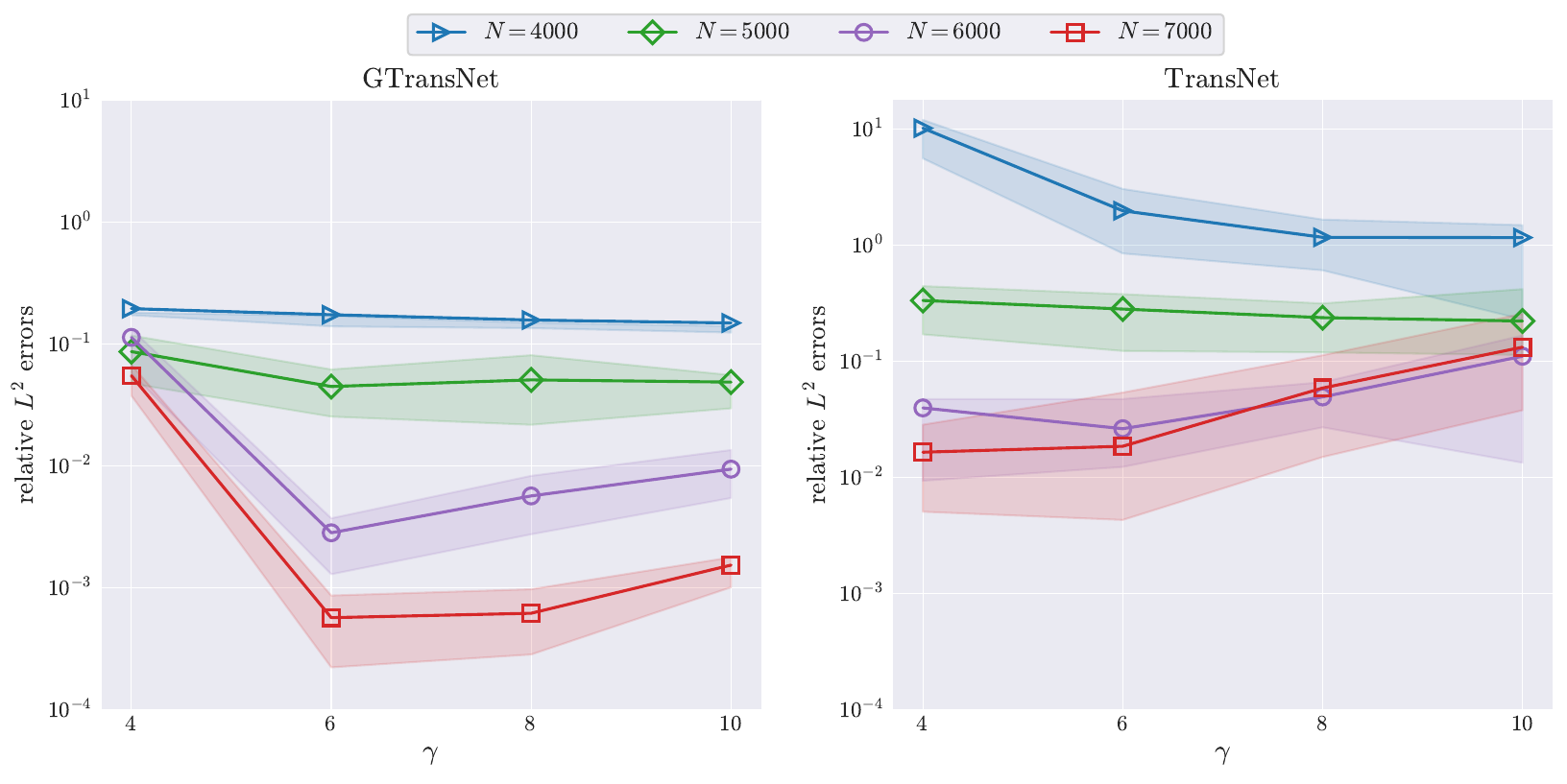}
    \vspace{-0.2cm}
    \caption{Plots of the relative $L^2$ errors of numerical solutions produced by the GTransNet method
    ({left}) and the TransNet method ({right}) method with different  values of the last hidden-layer neurons $N$ and the shape parameter $\gamma$ for the 3D steady-state Allen-Cahn equation  in Section 4.2.4.}\vspace{-0.3cm}
    \label{fig:Allen_convergence}
\end{figure}

\begin{figure}[ht!]
    \centering
    \includegraphics[width=0.72\linewidth]{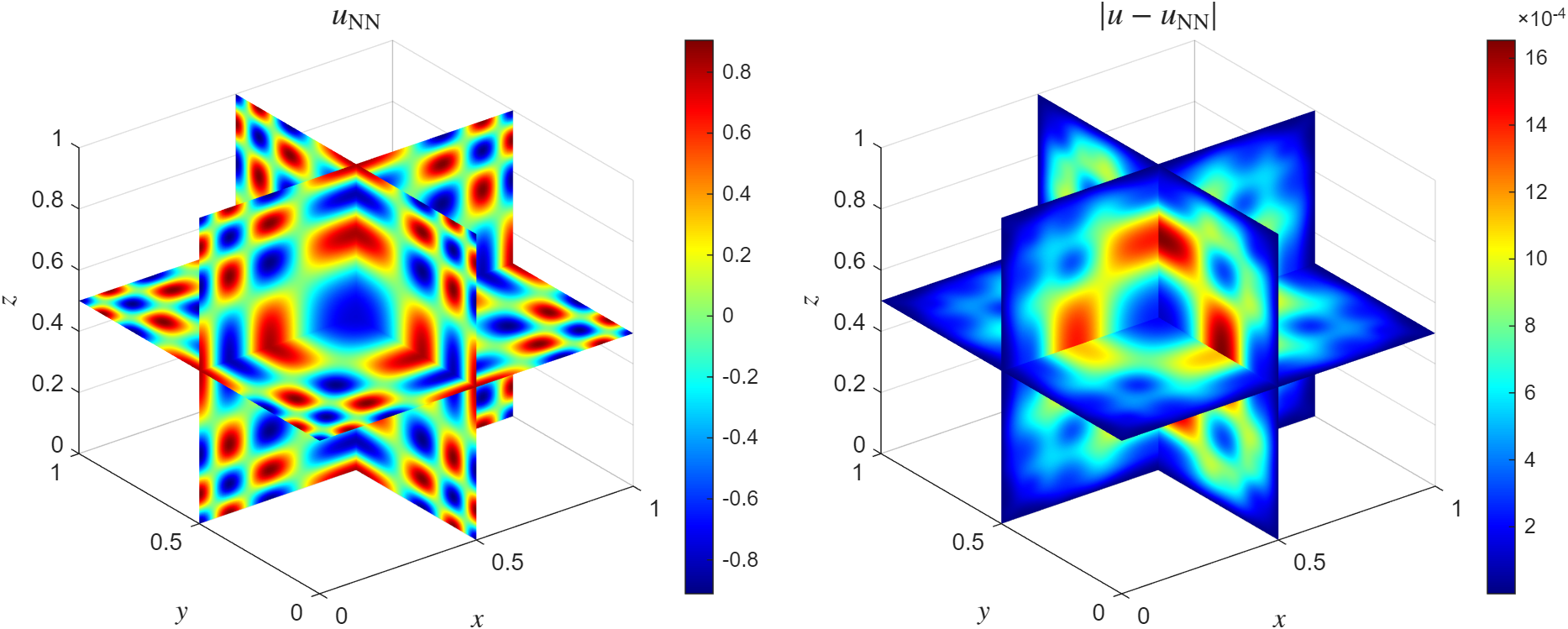}
    \vspace{-0.1cm}
    \caption{Numerical solution (left) and the corresponding absolute point-wise error (right) at the three cross sections $x=0.5$, $y=0.5$ and $z=0.5$  produced by the GTransNet method with $N=7000$ and $\gamma=6$ for the 3D steady-state Allen-Cahn equation in Section 4.2.4.}\vspace{-0.3cm}
    \label{fig:allen_cahn}
\end{figure}

In this example, GTransNet again produces much better numerical solutions than TransNet. Figure~\ref{fig:Allen_convergence} presents the relative $L^2$ errors of GTransNet and TransNet for the steady-state Allen–Cahn equation \eqref{eq:allen_cahn}, illustrating the effect of the shape parameter $\gamma$ and the performance under different numbers of neurons in the last hidden layer. We observe that a larger $N$ (e.g., $7000$) combined with a moderate $\gamma$ (e.g., $6$) yields the best accuracy for this problem among all tested settings of GTransNet.
Figure~\ref{fig:allen_cahn} further presents the predicted solution and the corresponding absolute pointwise error at the three cross sections $x = 0.5$, $y = 0.5$, and $z = 0.5$ produced by GTransNet with this setting, showing that the pointwise errors are approximately $\mathcal{O}(10^{-4})$ and remain low across the entire computational domain.

\section{Conclusion}\label{sec:5}
In this paper, we have proposed the GTransNet method for solving steady-state PDEs, which inherits the efficiency advantages of current popular shallow neural network–based PDE solvers by avoiding complex mesh generation in classical numerical schemes and the slow SGD training process in neural network methods. Our GTransNet can make use of fast least-squares solution techniques while overcoming the expressive power bottleneck of the TransNet method. Its ability to generate uniformly distributed high-frequency and multiscale features without excessive neuron expansion makes it an accurate, efficient and scalable computational framework for solving steady-state PDEs with  highly oscillatory solutions. Some future research directions can be explored along two main aspects. One is to investigate the optimization of the shape parameter $\gamma$ or $\{\gamma_m\}_{m=1}^{N_1}$ of the first hidden-layer neurons through adaptive algorithms, in order to further enhance the model’s performance across diverse problems. The other is to explore the combination of GTransNet with domain decomposition techniques to further improve computational efficiency.

\bibliographystyle{plain}
\bibliography{ref}
\end{document}